\documentclass[letterpaper,11pt,reqno]{amsart}
\usepackage{graphicx,amsmath,amssymb,amsthm, float}

\usepackage{srcltx}
\usepackage[latin1]{inputenc}
\usepackage[T1]{fontenc}

\usepackage[left=1in,right=1in, top=1in, bottom=1in]{geometry}

\newtheorem{X}{X}[section]
\newtheorem{corollary}[X]{Corollary}

\newtheorem{lemma}[X]{Lemma}
\newtheorem{hypothesis}[X]{Hypothesis}
\newtheorem{proposition}[X]{Proposition}
\newtheorem{theorem}[X]{Theorem}
\newtheorem{conjecture}[X]{Conjecture}

\theoremstyle{definition}
\newtheorem{remark}[X]{Remark}

\newcommand{\F}{\mathbb F}

\newcommand{\CC}{\mathbb{C}}

\newcommand{\wt}{\widetilde{w} \left( \frac dX \right)}
\newcommand{\wn}{w \left( \frac dX \right)}
\newcommand{\R}{\mathbb R}
\newcommand{\D}{\mathcal D}
\newcommand{\Q}{\mathbb Q}

\newcommand{\xf}{\frac{1}{W(X)}}

\newcommand{\col}{\: : \:}
\newcommand{\sumn}{\underset{(d,N_E)=1}{\sum \nolimits^{*}} w\left( \frac dX \right) }
\newcommand{\sumt}{\underset{(d,N_E)=1}{\sum \nolimits^{*}} \widetilde w\left( \frac dX \right) }
\newcommand{\sumtstar}{\underset{(d,N_E)=1}{\sum \nolimits^{*}} w\left( \frac dX \right) }

\newcommand{\sumnn}{\underset{(d,N_E)=1}{{\sum}^*}}
\newcommand{\sumtp}{\underset{\substack{(d,N_E)=1\\ p\mid d}}{\sum \nolimits^{*}} \widetilde w\left( \frac dX \right) }

\numberwithin{equation}{section}

\title[Low-lying zeros of elliptic curve $L$-functions]{Low-lying zeros of elliptic curve $L$-functions: \\Beyond the ratios conjecture}
\author{Daniel Fiorilli, James Parks and Anders S\"odergren}
\address{Department of Mathematics, University of Michigan, 530 Church Street, Ann Arbor, MI \newline
\rule[0ex]{0ex}{0ex}\hspace{8pt} 48109, USA}
\email{fiorilli@umich.edu}
\address{Department of Mathematics and Statistics, Concordia University, Montreal, QC, H3G\newline
\rule[0ex]{0ex}{0ex}\hspace{8pt} 1M8, Canada\newline
\rule[0ex]{0ex}{0ex}\hspace{8pt} \textit{Present address}: Department of Mathematics and Computer Science, University of \newline
\rule[0ex]{0ex}{0ex}\hspace{8pt} Lethbridge, 4401 University Drive, Lethbridge, AB, T1K 3M4, Canada}
\email{james.parks@uleth.ca}
\address{School of Mathematics, Institute for Advanced Study, Einstein Drive, Princeton, NJ\newline
\rule[0ex]{0ex}{0ex}\hspace{8pt} 08540, USA\newline
\rule[0ex]{0ex}{0ex}\hspace{8pt} \textit{Present address}: Department of Mathematical Sciences, University of Copenhagen,\newline
\rule[0ex]{0ex}{0ex}\hspace{8pt} Universitetsparken 5, 2100 Copenhagen \O, Denmark}
\email{sodergren@math.ku.dk} 

\thanks{The first author was supported by an NSERC Postdoctoral Fellowship. The second author was supported by a PIMS Postdoctoral Fellowship. The third author was supported by a Postdoctoral Fellowship from the Swedish Research Council, by the National Science Foundation under agreement No.\ DMS-1128155, as well as by a grant from the Danish Council for Independent Research and FP7 Marie Curie Actions-COFUND (grant id: DFF-1325-00058).}

\date{\today}
\begin{document}

\begin{abstract}
We study the low-lying zeros of $L$-functions attached to quadratic twists of a given elliptic curve $E$ defined over $\Q$. We are primarily interested in the family of all twists coprime to the conductor of $E$ and compute a very precise expression for the corresponding $1$-level density. In particular, for test functions whose Fourier transforms have sufficiently restricted support, we are able to compute the $1$-level density up to an error term that is significantly sharper than the square-root error term predicted by the $L$-functions Ratios Conjecture. 
\end{abstract}

\maketitle

\section{Introduction}\label{introduction}

The connection between zeros of $L$-functions and eigenvalues of random matrices first appeared in Montgomery's seminal paper on the pair correlation of zeros of the Riemann zeta function \cite{Mo}, where he proved that for suitably restricted test functions the pair correlation of the zeros of $\zeta(s)$ equals the pair correlation of the eigenvalues of random matrices from the Gaussian Unitary Ensemble (GUE). This work was later complemented by extensive numerical calculations of the zeros of $\zeta(s)$ by Odlyzko \cite{O1,O} that gave outstanding evidence for the agreement between local statistics of these zeros and the corresponding GUE statistics. It has also been shown that the $n$-level correlations of the zeros of $\zeta(s)$ agree with the corresponding GUE statistics, again for suitably restricted test functions \cite{He,RS1}.

In the work of Rudnick and Sarnak \cite{RS2}, it is shown that the $n$-level correlations of zeros of primitive automorphic $L$-functions all agree with the GUE statistics. However, it was predicted by Katz and Sarnak \cite{KS2,KS3} that by looking at low-lying zeros of families of $L$-functions, one should expect different statistics, which correspond to statistics of eigenvalues coming from scaling limits of certain compact Lie groups, specifically one of $U(N), O(N), SO(2N+1), SO(2N)$ and $Sp(2N)$.

Our purpose in the present paper is to study the low-lying zeros of the $L$-functions attached to the family of quadratic twists of a given elliptic curve $E$ over $\Q$. We assume that $E$ is given in global minimal Weierstrass form as
\begin{equation}\label{equation Weierstrass form}
E: y^2= x^3+ax+b,
\end{equation}
where $a,b \in \mathbb Z$. The discriminant of $E$ equals
\begin{align*}
\Delta_E:=-16(4a^3+27b^2)
\end{align*}
and is necessarily non-zero. We denote the conductor of $E$ by $N_E$ and recall that for $p>3$ we have $p\mid N_E$ if and only if $p \mid \Delta_E$. In general the conductor of an elliptic curve is a rather subtle object. However, for our purposes it will be enough to note that for all elliptic curves over $\Q$ the conductor is at least $11$; in particular we have $(N_E,0)>1$ (see, e.g., \cite{C}).

We now recall the definition of the $L$-function of $E$. The trace of the Frobenius endomorphism is given, for $p\nmid N_E$, by $a_p(E)= p+1-\#E_p(\F_p)$, where $\#E_p(\F_p)$ is the number of projective points on the reduction of $E$ modulo $p$. Extending the definition of $a_p(E)$ to the set of primes $p\mid N_E$ by setting
$$ a_p(E) := \begin{cases}
1 &\text{ if } E \text{ has split multiplicative reduction at }p, \\
-1 &\text{ if } E \text{ has non-split multiplicative reduction at }p, \\
0 &\text{ if } E \text{ has additive reduction at }p, \\
\end{cases} $$
the $L$-function of $E$ is defined as the Euler product
\begin{align}\label{Eulerproduct}
L(s,E):=&\prod_{p\mid N_E} \left( 1-\frac {a_p(E)}{p^{s+1/2}}\right)^{-1} \prod_{p\nmid N_E} \left( 1-\frac{a_p(E)}{p^{s+1/2}}+\frac 1{p^{2s}}\right)^{-1}\nonumber\\
=&\prod_{p} \left(1-\frac{\alpha_E(p)}{p^{s}}\right)^{-1}\left(1-\frac{\beta_E(p)}{p^{s}}\right)^{-1},\ \qquad \Re(s)>1.
\end{align}
Here, for all $p\nmid N_E$, $\alpha_E(p)$ and $\beta_E(p)$ are complex numbers satisfying $\beta_E(p)=\overline{\alpha_E(p)}$, $|\alpha_E(p)|=|\beta_E(p)|=1$ and $\alpha_E(p)+\beta_E(p)=a_p(E)/\sqrt p$. Moreover, in the remaining cases, that is when $p\mid N_E$, $\alpha_E(p)$ and $\beta_E(p)$ satisfy $\alpha_E(p)=a_p(E)/\sqrt p$ and $\beta_E(p)=0$. Thus $L(s,E)$ satisfies the Ramanujan-Petersson conjecture; in particular we have $|\alpha_E(p)|,|\beta_E(p)|\leq 1$ for all primes $p$. Note that with the above normalization the critical strip of $L(s,E)$ is $0<\Re(s)<1$. Expanding the product \eqref{Eulerproduct}, we define the sequence $\{\lambda_E(n)\}_{n=1}^{\infty}$ as the coefficients in the resulting Dirichlet series:
\begin{align*}
L(s,E)=\sum_{n=1}^{\infty}\frac{\lambda_E(n)}{n^s},\qquad \Re(s)>1.
\end{align*}
By the impressive work of Wiles \cite{W}, Taylor and Wiles \cite{TW}, and Breuil, Conrad, Diamond, and Taylor \cite{BCDT}, we know that there exists a cuspidal newform $f_E$ of weight $2$ and level $N_E$ such that $L(s,E)=L(s,f_E)$, that is, $L(s,E)$ is a modular $L$-function. In particular, it follows that $L(s,E)$ has an analytic continuation to the complex plane and that $L(s,E)$ satisfies the functional equation
\begin{align}\label{functionalequation}
\Lambda(s,E):=\left(\frac{\sqrt{N_E}}{2\pi}\right)^{s+\frac12}\Gamma(s+\tfrac12)L(s,E)=\epsilon_E\Lambda(1-s,E),
\end{align}
where $\epsilon_E=\pm1$ is the root number of $E$.

We are interested in the quadratic twists
\begin{equation}\label{equation quadratic twist}
E_d: dy^2= x^3+ax+b,
\end{equation}
of the fixed elliptic curve $E$.  It is clear that we can, by a change of variables, assume that $d$ is square-free. We furthermore restrict our attention to twists by integers $d$ satisfying $(d,N_E)=1$ and note that for such $d$ the conductor of $E_d$ equals $N_{E_d}=N_Ed^2$. We let $L(s,E_d)$ denote the $L$-function of $E_d$ defined by an Euler product as in \eqref{Eulerproduct}. As above, $L(s,E_d)$ is entire and satisfies the functional equation
\begin{align}\label{functionaltwist}
\Lambda(s,E_d):=\left(\frac{\sqrt{N_E}|d|}{2\pi}\right)^{s+\frac12}\Gamma(s+\tfrac12)L(s,E_d)=\epsilon_{E_d}\Lambda(1-s,E_d),
\end{align}
where the root number of $E_d$ satisfies $\epsilon_{E_d}=\epsilon_E \frac{d}{|d|}\chi_d(N_E)$ (cf.\ \cite[Prop.\ 14.20]{IK}; see also \cite[p.\ 538]{IK}). We let $\chi_d(n)$ denote the quadratic character defined in terms of the Kronecker symbol by
\begin{align*}
\chi_d(n):=\left( \frac dn\right).
\end{align*}
We will use the symbols $\chi_d(\cdot)$ and $(\tfrac d{\cdot})$ interchangeably throughout the paper. We stress that, as long as we keep $d$ square-free, $\chi_d$ is a  real primitive character (see \cite[Sect.\ 5]{D}). 

\begin{remark}
It is useful to note that $L(s,E_d)$ equals the Rankin-Selberg $L$-function
\begin{align*}
L(s,f_E\otimes \chi_d):=\sum_{n=1}^{\infty}\frac{\lambda_E(n)\chi_d(n)}{n^s},\qquad \Re(s)>1.
\end{align*}
\end{remark}

\begin{remark}
In what follows (cf.,\ e.g.,\ \eqref{equation one level density}) we will consider quadratic twists $E_d$ and their $L$-functions $L(s,E_d)$ also for non-square-free $d$. Writing $d'$ for the square-free part of $d$, note that  $L(s,E_{d})$ and $L(s,E_{d'})$ have the same nontrivial zeros. This observation will always allow us to reduce the study of quadratic twists by general $d$ to quadratic twists by square-free $d'$. 
\end{remark}

The $L$-functions coming from elliptic curves are in many ways analogous to the Riemann zeta function. In particular they are expected to satisfy the following Riemann hypothesis. 

\begin{hypothesis}[Elliptic Curve Riemann Hypothesis, ECRH] For any elliptic curve $E$ over $\mathbb Q$, the nontrivial zeros of $L(s,E)$ have real part equal to $\frac12$.
\end{hypothesis}

 Throughout this paper, $\phi$ will denote an even Schwartz test function satisfying $\phi(\mathbb R) \subset \mathbb R$. The Fourier transform of $\phi$ is defined as
$$ \widehat \phi(\xi):= \int_{\mathbb R}  \phi(t) e^{-2\pi i \xi t} dt. $$
The quantity we are interested in is the weighted $1$-level density of low-lying zeros of the family of $L$-functions attached to the quadratic twists $E_d$ of the fixed elliptic curve $E$. Given a (large) positive number $X$ and a test function $\phi$, we introduce the $1$-level density for the single $L$-function $L(s,E_d)$ as the sum
\begin{align*}
D_X(E_d;\phi):=\sum_{\gamma_d} \phi\left( \gamma_d\frac{  L}{2\pi} \right),
\end{align*}
with $\gamma_d:= -i(\rho_d-\frac12)$, where $\rho_d$ runs over the nontrivial zeros of $L(s,E_d)$. Moreover, $L$ is defined by
\begin{align}\label{Ldefinition}
L:=\log\left(\frac{N_E X^2}{(2\pi e)^2}\right).
\end{align}
Note that $L$ is chosen so that for $d\approx X$, the sequence $\displaystyle{\frac{ \gamma_d L}{2\pi}}$ of normalized low-lying zeros arising from (say) $\gamma_d\leq1$ has essentially constant mean spacing (recall that
by a Riemann-von Mangold type theorem as in for example \cite[Thm.\ 5.8]{IK} $L(s,E_d)$ has approximately $\log (N_E d^2 /(2\pi e)^2)/2\pi$ zeros in the region $0<\Im(\rho_d) \leq 1$). We further remark that ECRH asserts that $\gamma_d\in \mathbb R$.\footnote{It is not essential for our results to assume ECRH at this point. However, ECRH is of course crucial for enabling a spectral interpretation of our results.}

The quantity $D_X(E_d;\phi)$ is very hard to understand for individual elliptic curves. However, when we consider averages over families of quadratic twists the situation becomes much more tractable. To allow a maximally detailed analysis of the the $1$-level density, we first
consider the family\footnote{This family contains elliptic curves with both signs in the functional equation. We choose not to separate the family into even and odd signs in order to keep the statement of our main results as concise as possible. A similar analysis can give the corresponding results also for the even and odd families.} $\{E_d \col (d,N_E)=1\}$ which clearly contains an abundance of repetitions.\footnote{Allowing repetitions in the family in order to make the analysis manageable is not a new strategy; cf., e.g., \cite{Y2,FM}.} Hence we introduce the following weighted $1$-level density
\begin{equation}\label{equation one level density}
\D(\phi;X) := \frac 1{W(X)} \sum_{(d,N_E)=1}w\left( \frac dX \right)D_X(E_d;\phi),
\end{equation}
where $w(t)$ is an even nonnegative Schwartz test function having positive total mass (which morally restricts the sum to $d\ll X$) and
\begin{align}\label{W-sum}
W(X):=\sum_{(d,N_E)=1}w\left( \frac dX \right).
\end{align}
Note that any given zero occurring in \eqref{equation one level density} will be repeated infinitely many times. However, since $\phi$ and $w$ are rapidly decaying, most zeros will not be given a large total weight in the outer sum. Indeed, a zero that is given a large total weight necessarily originates from a curve with small conductor, but such zeros are on average quite far from the real line and thus cannot give a large contribution to the $1$-level density.

Quantities analogous to $\D(\phi;X)$ have been studied by Goldfeld \cite{Go}, Brumer \cite{Br}, Heath-Brown \cite{HB} and Young \cite{Y2}, in order to obtain conditional bounds on the average rank of elliptic curves in certain families. One can reinterpret their results as asymptotic estimates for $\D(\phi;X)$ when the support of $\widehat \phi$ is appropriately restricted; the larger the allowable support is, the better the resulting upper bound on the average rank becomes.

To predict an asymptotic for $\D(\phi;X)$, Katz and Sarnak \cite{KS2,KS3} associate a given (natural) family $\mathcal F$ of $L$-functions defined over a number field with a corresponding family of $L$-functions defined over a suitable function field. By an analysis of the function field family, they predict that the low-lying zeros of the $L$-functions in $\mathcal F$ behave like the eigenvalues near $1$ in a related compact Lie group $G(\mathcal F)$ of either unitary, orthogonal or symplectic matrices. In our case the symmetry group $G(\mathcal F)$ is $O(N)$, and the Katz-Sarnak prediction takes the form 
\begin{equation}
\lim_{X \rightarrow \infty}\D(\phi;X)=\widehat \phi(0)+ \frac {\phi(0)}2  = \int_{\mathbb R} \phi(t) \mathcal W_{O}(t) dt,
\label{equation Katz-Sarnak}
\end{equation}
where $\mathcal W_O:= 1+\frac 12 \delta_0$. The analogous prediction on a closely related family has been checked by Katz and Sarnak \cite{KS1}, for a restricted class of test functions. 

The Katz-Sarnak prediction on $\D(\phi;X)$ is given in terms of statistics of random matrices, and one can ask whether random matrix theory can 
predict other features of zeros of $L$-functions, such as possible lower order terms in \eqref{equation Katz-Sarnak}. It turns out that for
test functions $\phi$ whose Fourier transform has restricted support, Young \cite{Y1} has shown that, in certain families of elliptic curves, lower order terms of order $(\log X)^{-1}$ do exist in the $1$-level density. Moreover, these terms cannot be explained using random matrix theory.  Such lower order terms have also been found in families of quadratic twists of a fixed elliptic curve \cite{HMM}. The limitations of random matrix theory for making predictions on statistics of $L$-functions have also been observed in other contexts, most notably in predictions for moments \cite{KeS,KeS2}.

An extremely powerful conjecture was put forward by Conrey, Farmer and Zirnbauer \cite{CFZ}, which predicts estimates 
for averages of quotients of (products of) $L$-functions evaluated at certain values. A variant of this conjecture implies a formula for $\D(\phi;X)$ which contains the Katz-Sarnak prediction, lower order terms and an error term of size at most $O_{\varepsilon}(X^{-1/2+\varepsilon})$ (see \cite{HKS}). Other variants of the conjecture imply very precise estimates for many other $L$-function statistics \cite{CS}.

The Ratios Conjecture's prediction in our family contains the following modified weight function, on which we will expand in Section \ref{section prelim}. Given the even nonnegative Schwartz weight $w(t)$, we define
\begin{align}\label{wtildeintro}
 \widetilde w (x)= \widetilde w_E (x) := \sum_{ \substack{n\geq 1 \\ (n,N_E)=1}} w(n^2x). 
\end{align}
Moreover, throughout this paper the symbol $\sum_d^*$ indicates that the sum is restricted to square-free values of $d$. We now state a precise consequence of the Ratios Conjecture (Conjecture \ref{ratiosconjecture}). The proof of Theorem \ref{theratios} will be given in Appendix \ref{Appendix A}.

\begin{theorem} \label{theratios}
Fix $\varepsilon>0$. Let $E$ be an elliptic curve defined over $\Q$ with conductor $N_E$. Let $w$ be a nonnegative Schwartz function on $\R$ which is not identically zero and let $\phi$ be an even Schwartz function on $\R$ whose Fourier transform has compact support. Assuming GRH\footnote{In this paper, GRH denotes the Riemann Hypothesis for $\zeta(s), L(s,E)$ and $L(s,\text{Sym}^2E)$ for every elliptic curve $E$ over $\Q$.} and Conjecture \ref{ratiosconjecture} (the Ratios Conjecture for our family), the $1$-level density for the zeros of the family of $L$-functions attached to the quadratic twists of $E$ coprime to $N_E$ is given by
\begin{align}
\mathcal D(\phi;X)=& \frac {\widehat \phi (0)}{LW(X)} \sumt\log\bigg(\frac{N_{E}d^2}{(2\pi)^2}
\bigg) + \frac{1}{2\pi}\int_{\mathbb R}
\phi\left(\frac{tL}{2\pi}\right)\bigg[\frac{\Gamma'(1+it) }{\Gamma(1+it) } \notag\\
 & +\frac{\Gamma'(1-it)}{\Gamma(1-it)} + 2 \bigg(-\frac{\zeta'(1+2it)}{\zeta(1+2it)} +\frac{L'\left(1+2it,{\rm Sym}^2E\right)}{L\left(1+2it,{\rm Sym}^2E\right)}+A_{\alpha,E}(it,it)\bigg)- \frac{1}{it} \bigg]dt \label{equation ratios conj}\\
&+  \frac{\phi(0)}{2} + O_{\varepsilon}\big(X^{-1/2+\varepsilon}\big), \notag
\end{align}
where $^*$ indicates that we are summing over square-free $d$, the functions $L$, $W$ and $\widetilde w$ are defined by \eqref{Ldefinition}, \eqref{W-sum} and \eqref{wtildeintro} respectively, $L\left(s,{\rm Sym}^2E\right)$ is the symmetric square $L$-function of $E$ (cf.\ \cite{Sh}), and the function $A_{\alpha,E}$ is defined by \eqref{AALPHAE} (see also \eqref{vnmid}, \eqref{vmid}, \eqref{defofy} and \eqref{defnofae}). The implied constant in the error term depends on $E$, $\phi$ and $w$.
\end{theorem}

\begin{remark}
One can, for any $K\in \mathbb N$, rewrite \eqref{equation ratios conj} in the form\footnote{Note that the first two terms on the right-hand side of \eqref{equation descending powers} coincide with the Katz-Sarnak prediction.}
\begin{equation}
\D(\phi;X)= \widehat \phi (0) + \frac{\phi(0)}2 + \sum_{j=1}^{K-1} \frac{c_j(\phi,w,E)}{(\log X)^j} +O_{\phi,w,E,K} \left( \frac 1{(\log X)^K} \right),
\label{equation descending powers}
\end{equation} 
for some constants $c_j(\phi,w,E)$. Indeed, Lemma \ref{lemma logd} implies that the first term on the right-hand side of \eqref{equation ratios conj} is of the desired form. 
As for the second, making the change of variables $u= tL/2\pi$, truncating the resulting integral at the points $u=\pm \sqrt L$ and expanding the expression in square brackets into Taylor series around zero gives the desired expansion. 
\end{remark}

One of our main objectives is to give an estimate for $\D(\phi;X)$ with an error term of size at most $O_{\varepsilon}(X^{-1/2+\varepsilon})$, for test functions whose Fourier transforms have small support. Our first main theorem shows that we can obtain such an estimate with an error term that is significantly sharper than the error term appearing in the Ratios Conjecture's prediction (cf. Theorem \ref{theratios}). 

\begin{theorem}\label{main theorem}
Fix $\varepsilon>0$. Let $E$ be an elliptic curve defined over $\Q$ with conductor $N_E$. Let $w$ be a nonnegative Schwartz function on $\R$ which is not identically zero and let $\phi$ be an even Schwartz function on $\R$ whose Fourier transform satisfies $\sigma=$sup$(\text{supp}\widehat \phi)<\frac 12$. Then the $1$-level density for the zeros of the family of $L$-functions attached to the quadratic twists of $E$ coprime to $N_E$ is given by
\begin{align}
\begin{split}
\D(\phi;X)&=\frac {\widehat \phi(0)}{LW(X)}  \sumt \log \bigg(\frac{N_{E} d^2}{(2\pi)^2}\bigg)-\frac{2} L\int_0^{\infty} \left( \frac{\widehat \phi(x/L) e^{-x} }{1-e^{-x}} - \widehat \phi(0) \frac{e^{-x}}x \right)dx\\
&-\frac {2}{L}\sum_{\ell=1}^\infty \sum_{p} \frac{\left(\alpha_E(p)^{2\ell}+\beta_E(p)^{2\ell}\right) \log p}{p^{\ell}} \widehat \phi\left( \frac{2\ell\log p}{L} \right) \left( 1+\frac{\psi_{N_E}(p)}{p}\right)^{-1}+O_{\varepsilon}\big(X^{\eta(\sigma)+\varepsilon}\big),
\end{split}
\label{equation main thm}
\end{align}
where $^*$ indicates that we are summing over square-free $d$, $\psi_{N_E}$ is the principal Dirichlet character modulo $N_E$, the functions $L$, $W$ and $\widetilde w$ are defined by \eqref{Ldefinition}, \eqref{W-sum} and \eqref{wtildeintro} respectively, and 
$$
\eta(\sigma) =\begin{cases}
-1+2\sigma &\text{ if } \frac 1{4m+2} \leq  \sigma < \frac 1{4m+1}, \\
-\frac{4m-1}{4m+1} &\text{ if } \frac 1{4m+1} \leq  \sigma < \frac 1{4m-2},
\end{cases}
$$
for each $m\geq 1$. The implied constant in the error term depends on $E$, $\phi$ and $w$.
\end{theorem}

The techniques used to prove Theorem \ref{main theorem} are inspired by the work of Katz and Sarnak \cite{KS1}. The main tools we use, which were pioneered by Iwaniec \cite{I} in this context, are Poisson summation and the P\'olya-Vinogradov inequality. The key to obtaining an error term sharper than the Ratios Conjecture's prediction is to allow repetitions in our family, and to use the smooth cutoff function~$w$. 

\begin{remark}
Theorem \ref{extendedoneleveldensityresult} shows that the sum of the second and third terms appearing on the right-hand side of \eqref{equation main thm} matches the sum of the second and third terms appearing on the right-hand side of \eqref{equation ratios conj}, up to an error which is at most $O_{\varepsilon}(X^{-1+\varepsilon})$. Therefore, Theorem \ref{main theorem} agrees with the Ratios Conjecture's prediction, but is even more precise. This result should be compared with the main theorem of \cite{FM}, in which the authors obtain an estimate for the $1$-level density in the family of all Dirichlet $L$-functions, which is more precise than the Ratios Conjecture's prediction.
\end{remark}

\begin{remark}
It is possible to improve the estimate in Theorem \ref{main theorem} in certain ranges of $\sigma$ by using Burgess's bound. We have chosen to carry out this improvement in a separate paper \cite{FPS}. 
\end{remark}

In the next theorem we show that ECRH implies a formula for $\D(\phi;X)$ with a sharper error term, which in particular doubles the allowable support for $\widehat \phi$. 

\begin{theorem}\label{second main theorem}
Fix $\varepsilon>0$. Let $E$ be an elliptic curve defined over $\Q$ with conductor $N_E$. Let $w$ be a nonnegative Schwartz function on $\R$ which is not identically zero and let $\phi$ be an even Schwartz function on $\R$ whose Fourier transform satisfies $\sigma=$sup$(\text{supp}\widehat \phi)<1$. Then, assuming ECRH, the $1$-level density for the zeros of the family of $L$-functions attached to the quadratic twists of $E$ coprime to $N_E$ is given by
\begin{align}
\begin{split}
\D(\phi;X)&=\frac {\widehat \phi(0)}{LW(X)}  \sumt \log \bigg(\frac{N_{E} d^2}{(2\pi)^2}\bigg)-\frac{2} L\int_0^{\infty} \left( \frac{\widehat \phi(x/L) e^{-x} }{1-e^{-x}} - \widehat \phi(0) \frac{e^{-x}}x \right)dx\\
&-\frac {2}{L}\sum_{\ell=1}^\infty \sum_{p} \frac{\left(\alpha_E(p)^{2\ell}+\beta_E(p)^{2\ell}\right) \log p}{p^{\ell}} \widehat \phi\left( \frac{2\ell\log p}{L} \right) \left(1+\frac{\psi_{N_E}(p)}{p}\right)^{-1}+O_{\varepsilon}\big(X^{\theta(\sigma)+\varepsilon}\big),
\end{split}
\label{equation second main thm}
\end{align}
where $^*$ indicates that we are summing over square-free $d$, $\psi_{N_E}$ is the principal Dirichlet character modulo $N_E$, the functions $L$, $W$ and $\widetilde w$ are defined by \eqref{Ldefinition}, \eqref{W-sum} and \eqref{wtildeintro} respectively, and 
$$
\theta(\sigma) =\begin{cases}
-1+\sigma &\text{ if } \frac 1{4m+2} \leq  \sigma < \frac 1{4m}, \\
-1+\frac{1}{4m} &\text{ if } \frac 1{4m} \leq  \sigma < \frac 1{4m-2}, \\
-1+\sigma & \text{ if } \frac 12\leq \sigma <1.
\end{cases}
$$
The implied constant in the error term depends on $E$, $\phi$ and $w$.
\end{theorem}

In Figure \ref{figure comp}, we compare the exponent of $X$ in the error terms of Theorems \ref{main theorem} and \ref{second main theorem} by plotting $\eta(\sigma)$ and $\theta(\sigma)$ as functions of $\sigma=$sup$(\text{supp}\widehat \phi)$.

Finally, we study the weighted $1$-level density averaged over square-free values of $d$:
\begin{equation}
\D^*(\phi;X) := \frac 1{W^*(X)} \sumtstar D_X(E_d;\phi),\label{equation one level density star}
\end{equation}
where
\begin{equation}
W^*(X):=\sumtstar
\label{Wstar-sum}.
\end{equation}
This quantity is more natural to study than $\mathcal D(\phi;X)$, since there are no repetitions. However, the estimate we obtain for the error term is weaker, and we are not able to surpass the Ratios Conjecture's prediction in this case.

\begin{theorem}\label{third main theorem}
Fix $\varepsilon>0$. Let $E$ be an elliptic curve defined over $\Q$ with conductor $N_E$. Let $w$ be a nonnegative Schwartz function on $\R$ which is not identically zero and let $\phi$ be an even Schwartz function on $\R$ whose Fourier transform satisfies $\sigma=$sup$(\text{supp}\widehat \phi)<1$. Then, assuming the Riemann Hypothesis (RH) and ECRH, the $1$-level density for the zeros of the family of $L$-functions attached to the square-free quadratic twists of $E$ coprime to $N_E$ is given by
\begin{align*}
\D^*(\phi;X)&=\frac {\widehat \phi(0)}{LW^{*}(X)}  \sumtstar \log \bigg(\frac{N_{E} d^2}{(2\pi)^2}\bigg)-\frac{2} L\int_0^{\infty} \left( \frac{\widehat \phi(x/L) e^{-x} }{1-e^{-x}} - \widehat \phi(0) \frac{e^{-x}}x \right)dx\nonumber\\
&-\frac {2}{L}\sum_{\ell=1}^\infty \sum_{p} \frac{\left(\alpha_E(p)^{2\ell}+\beta_E(p)^{2\ell}\right) \log p}{p^{\ell}} \widehat \phi\left( \frac{2\ell\log p}{L} \right) \left( 1+\frac {\psi_{N_E}(p)}p \right)^{-1} +O_{\varepsilon}\big(X^{\frac{\sigma-1}2+\varepsilon}\big),
\end{align*}
where $^*$ indicates that we are summing over square-free $d$ and the functions $L$ and $W^{*}$ are defined by \eqref{Ldefinition} and \eqref{Wstar-sum}  respectively. The implied constant in the error term depends on $E$, $\phi$ and $w$.
\end{theorem}

\begin{remark}
One can remove the assumption of RH in Theorem \ref{third main theorem} by proving an unconditional version of Lemma \ref{lemma count of squarefree} with a weaker error term. Note that this modification does not affect the overall result. 
\end{remark}

\begin{remark}
Using similar techniques to those used to obtain Theorem \ref{third main theorem} but splitting the family according to the sign of the functional equation, one can improve both the allowable support of $\widehat \phi$ and the quality of the error term in the main theorem\footnote{Although in \cite[Theorem 1.1]{HMM} the authors claim an error term of $X^{-\frac{1-\sigma}2} (\log X)^6$, their proof only produces the weaker error term $X^{-\frac 12+\sigma} (\log X)^6$, which yields a nontrivial result for $\sigma<\tfrac 12$. In \cite[(2.35)]{HMM}, the restriction on the sum over primes should read $p^{2\ell+1} \leq X^{2\sigma}$, since in \cite[(2.31)]{HMM}, $\widehat g(\log p^{2k+1}/2L)$ is zero outside this range. Accounting for this, \cite[(2.35)]{HMM} results in the error term $X^{-\frac 12+\sigma} (\log X)^6$. Note also that the main term in \cite[Theorem 1.1]{HMM} is not correct as stated. Indeed, the third term on the right-hand side of \cite[(1.5)]{HMM} is the integral of a function which has a simple pole on the contour of integration.} of \cite{HMM}.
\end{remark}

\begin{figure}
\label{figure comp}
\begin{center}
\includegraphics[scale=.45]{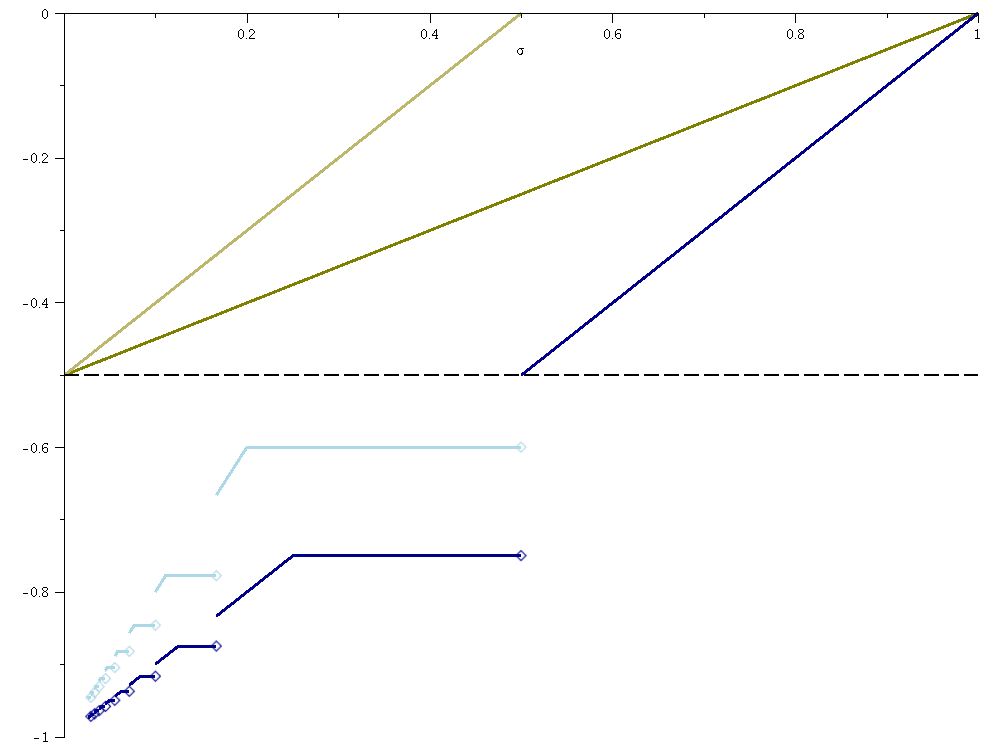}
\end{center}
\caption{A comparison of our unconditional results (light blue), our results under ECRH (dark blue), our results for $\mathcal D^*(\phi;X)$ (dark green), the results of \cite{HMM} (light green) and the Ratios Conjecture's prediction (dashed line).}
\end{figure}

\subsection{Acknowledgments}
The present work was initiated during the American Mathematical Society's  MRC program \emph{Arithmetic Statistics} in Snowbird, Utah, under the supervision of Nina Snaith. We are grateful to the AMS and the organizers of the program for financial support and encouragement. In particular we are grateful to Brian Conrey and Nina Snaith for inspiring discussions and helpful remarks. We also thank David Zywina for enlightening discussions.

\section{Preliminaries}
\label{section prelim}

We begin this section by showing that the Mellin Transform $\mathcal M w(s)$ has very nice analytical properties, which will be useful later. 
Recall that $w: \mathbb R \longrightarrow \mathbb R$ is a fixed nonnegative even Schwartz function which is not identically zero.

\begin{lemma}
\label{lemma fast decay}
The Mellin transform
$$ \mathcal M w(s):= \int_{0^+}^{\infty} x^{s-1} w(x) dx $$
is a holomorphic function of $s=\sigma+it$ except for possible simple poles at non-positive integers, and satisfies the bound
\begin{equation}
 \mathcal M w(s) \ll_{n,w} (1+|t|)^{-n}
 \label{equation bound on mellin}
\end{equation}
for any fixed $n \in \mathbb N$, uniformly for $\sigma$ in any compact subset of $\mathbb R$, provided $s$ is bounded away from the set $\{0,-1,-2,\dots \}$.
\end{lemma}

\begin{proof}
Note first that since $w$ is Schwartz, the integral $\int_{0^+}^{\infty} x^{s-1} w(x) dx$ converges absolutely and uniformly on any compact subset of $\{s \in \mathbb C : \Re(s) > 0 \} $. To give an analytic continuation of $\mathcal M w(s)$, we integrate by parts:
$$ \mathcal M w(s) = \frac{x^s w(x)}{s} \Bigg|_{0^+}^{\infty} - \frac 1s \int_{0^+}^{\infty} x^{s} w'(x) dx = - \frac 1s \int_{0^+}^{\infty} x^{s} w'(x) dx,$$
which by absolute and uniform convergence of the integral shows that $\mathcal M w(s)$ is a holomorphic function for $\sigma>-1$ except possibly for a simple pole at $s=0$. Iterating this process $n$ times we obtain the formula
$$ \mathcal M w(s)  = \frac{(-1)^n}{s(s+1)\cdots(s+n-1)}  \int_{0^+}^{\infty} x^{s+n-1} w^{(n)}(x) dx.$$
This shows that \eqref{equation bound on mellin} holds, and that $\mathcal M w(s)$ extends to a holomorphic function on $\mathbb C$ except for possible simple poles at $s=0,-1,-2,\ldots$. Here we used that the integral $\int_{0^+}^{\infty} x^{s+n-1} w^{(n)}(x) dx$ converges absolutely and uniformly on compact subsets of $\Re(s)>-n$, due to the fact that $w$ is Schwartz.
\end{proof}

\begin{remark}
The proof of Lemma \ref{lemma fast decay} shows that $\mathcal M w(s)$ is holomorphic at $s=-n$ when $w^{(n)}(0)=0$, and has a simple pole at this point otherwise.
\end{remark}

\begin{lemma}
\label{lemma mellin of tilde}
Define the even smooth function $\widetilde w: \mathbb R \setminus \{0\} \rightarrow \mathbb R $ by
\begin{equation}
 \widetilde w (x)= \widetilde w_E (x) := \sum_{ \substack{n\geq 1 \\ (n,N_E)=1}} w(n^2x). \label{eq:wtildedef}
\end{equation}
Then $\widetilde w(x)$ decays rapidly as $x\rightarrow \infty$, and we have that
$$ \mathcal M \widetilde w (s) = \prod_{p\mid N_E} \left( 1-\frac 1{p^{2s}} \right) \zeta(2s) \mathcal M w(s). $$
\end{lemma}
\begin{remark}
Note that $\widetilde w(x)$ blows up near $x=0$. Indeed, it follows from \eqref{eq:wtildedef} and our assumptions on $w(x)$ that $\widetilde w(x)\asymp x^{-1/2}$ as $x\to 0$.
\end{remark}

\begin{proof}[Proof of Lemma \ref{lemma mellin of tilde}]
For $\Re(s)$ large enough, we have
\begin{align*}
\int_{0^+}^{\infty} x^{s-1} \widetilde w (x) dx &= \sum_{ \substack{n\geq 1 \\ (n,N_E)=1}} \int_{0^+}^{\infty} x^{s-1} w(n^2x) dx \\
&= \sum_{ \substack{n\geq 1 \\ (n,N_E)=1}} n^{-2s}  \int_{0^+}^{\infty}  u^{s-1} w(u) du =\prod_{p\mid N_E} \left( 1-\frac 1{p^{2s}} \right) \zeta(2s) \mathcal M w(s).
\end{align*}
The result follows by analytic continuation.
\end{proof}

\subsection{Weighted character sums}

The following estimate is central in our analysis of $\mathcal D(\phi;X)$.

\begin{lemma}\label{lemmaweighted}
Fix $n \in \mathbb N$ and $\varepsilon>0$. We have the estimate
\begin{multline*} \sumt  \left( \frac dn\right) =\kappa(n)X  \widehat w(0) \prod_{p\mid n} \left( 1+\frac 1 p\right)^{-1} \prod_{\substack{p\mid N_E}} \left( 1-\frac{1}{p} \right)  \prod_{\substack{p\mid (n,N_E)}} \left( 1+\frac{1}{p} \right) \\+O_{\varepsilon,w}\big(|n|^{\frac 12-\frac{\kappa(n)}2+\varepsilon}(XN_E)^{\varepsilon}\big), \end{multline*}
where
$$ \kappa(n) := \begin{cases}
1 &\text{ if } n = \square, \\
0 &\text{ otherwise.}
\end{cases}$$
\end{lemma}

\begin{remark} \label{whatiswx}
In particular, taking $n=1$ in Lemma \ref{lemmaweighted} gives 
$$ W(X)= \sumt = X  \widehat w(0)  \prod_{\substack{p\mid N_E}} \left( 1-\frac{1}{p} \right) + O_{\varepsilon,w}\big((XN_E)^{\varepsilon}\big), $$
an estimate which will be useful in the proof of Theorem \ref{main theorem} and in Appendix \ref{Appendix A}.
\end{remark}

\begin{remark}\label{remarkgeneralinteger}
Lemma \ref{lemmaweighted} applies equally well when $N_E$ is replaced by any nonzero integer (not necessarily the conductor of an elliptic curve).
\end{remark}

\begin{proof}[Proof of Lemma \ref{lemmaweighted}]
First note that $\widetilde w(t)$ is even, and hence
\begin{equation}
\sumt \left( \frac dn\right) = \left(1+\left( \frac{-1}n \right)\right)\underset{\substack{ d\geq 1 \\ (d,N_E)=1}}{{\sum}^*}  \wt \left( \frac dn\right).
\label{eq:seperation in two sums}
\end{equation}
Using Mellin inversion, we write
\begin{align*}
I_n(X):&=\sum_{ \substack{d\geq 1 \\ (d,N_E)=1}} \mu^2(d) \wt \left( \frac dn\right) \\
 &= \frac 1{2\pi i } \int_{\Re(s)=2} \sum_{\substack{d\geq 1 \\ (d,N_E)=1}} \frac{\mu^2(d) \left(\frac dn \right)}{d^s} X^s \mathcal M \widetilde w(s) ds \\
&= \frac 1{2\pi i } \int_{\Re(s)=2} \prod_{p\nmid N_E} \bigg(1+ \frac{\left(\frac pn \right)}{p^s} \bigg) \prod_{p\mid N_E} \left( 1-\frac 1{p^{2s}} \right) \zeta(2s) \mathcal M w (s)X^s ds \hspace{1cm} \text{(by Lemma \ref{lemma mellin of tilde})}\\
&= \frac 1{2\pi i } \int_{\Re(s)=2} \prod_{p\mid N_E} \bigg(1+ \frac{\left(\frac pn \right)}{p^s} \bigg)^{-1} \frac{L(s,\left(\frac{\cdot}n \right))}{L(2s,\left(\frac{\cdot}n \right)^2)} \prod_{p\mid N_E} \left( 1-\frac 1{p^{2s}} \right)\zeta(2s) \mathcal M w (s)X^s ds\\
&=  \frac 1{2\pi i } \int_{\Re(s)=2} L\big(s,\left(\frac{\cdot}n \right)\big) \prod_{p\mid N_E} \bigg(1+ \frac{\left(\frac pn \right)}{p^s} 
\bigg)^{-1}\left( 1-\frac 1{p^{2s}} \right) \prod_{p\mid n} \left( 1-\frac 1{p^{2s}} \right)^{-1}  \mathcal Mw(s) X^s ds.
\end{align*}

We first consider the case when $n$ is not a square. In this case $L(s,\left(\frac{\cdot}n \right))$ is holomorphic at $s=1$, and thus we can shift the contour of integration to the left:
$$ I_n(X) = \frac 1{2\pi i } \int_{\Re(s)=\varepsilon} L\big(s,\left(\frac{\cdot}n \right)\big) \prod_{p\mid N_E} \bigg(1+ \frac{\left(\frac pn \right)}{p^s} \bigg)^{-1}\left( 1-\frac 1{p^{2s}} \right) \prod_{p\mid n} \left( 1-\frac 1{p^{2s}} \right)^{-1}  \mathcal Mw(s) X^s ds. $$
Here $0<\varepsilon < 1/4$ is fixed (note that the integrand might have poles on the line $\Re(s)=0$). We then apply the convexity bound \cite[(5.20)]{IK}, which for non-principal $\chi$ reads
$$ L(s,\chi) \ll_{\varepsilon} (q(|s|+1))^{\frac{1-\Re(s)}2+\frac{\varepsilon}2} \hspace{1cm} (0\leq \Re(s) \leq 1),  $$
where $q$ is the conductor of $\chi$. Combining this with Lemma \ref{lemma fast decay} yields the bound
\begin{align*}  I_n(X) &\ll_{\varepsilon,w} \int_{ \mathbb R} (|n|(|t|+1))^{\frac{1-\varepsilon}2+\frac{\varepsilon}2} \prod_{p\mid N_E} \bigg(1- \frac{1}{p^{\varepsilon}} \bigg)^{-1}\left( 1+\frac 1{p^{2\varepsilon}} \right) \prod_{p\mid n} \left( 1-\frac 1{p^{2\varepsilon}} \right)^{-1} (|t|+1)^{-2} X^{\varepsilon} dt  \\
& \ll_{\varepsilon} |n|^{\frac 12+ \varepsilon} N_E^{\varepsilon} X^{\varepsilon}.
\end{align*}
Indeed, we have the bound
$$ \prod_{p\mid N_E} \bigg(1- \frac{1}{p^{\varepsilon}} \bigg)^{-1} \ll_{\varepsilon} \prod_{\substack{p\mid N_E \\ p> 2^{1/\varepsilon}}} \bigg(1- \frac{1}{p^{\varepsilon}} \bigg)^{-1} \leq \prod_{\substack{p\mid N_E \\ p> 2^{1/\varepsilon}}} 2 \ll_{\varepsilon} N_E^{\varepsilon/2}, $$
and the two other products over primes are bounded in a similar fashion.
The proof of this case is completed by combining these estimates with \eqref{eq:seperation in two sums}.

As for the case where $n$ is a square, we again shift the contour of integration to the left, picking up a residue at $s=1$. Note that in this case
$$L\big(s,\left(\frac{\cdot}n  \right)\big) = \prod_{p\mid n} \left( 1-\frac 1{p^s} \right) \zeta(s),  $$
and hence the contribution of the pole at $s=1$ is given by
$$ \prod_{p\mid n} \left( 1+\frac 1 p\right)^{-1} \prod_{p\mid N_E} \left( 1+\frac{\left( \frac pn\right)}{p} \right)^{-1} \left( 1-\frac 1{p^{2}} \right)  \mathcal M w(1) X.  $$
By Lemma \ref{lemma fast decay}, the shifted integral is
\begin{align*}  &\ll_{\varepsilon,w} \int_{ \mathbb R} (|t|+1)^{\frac{1-\varepsilon}2+\frac{\varepsilon}2}  \prod_{p\mid n} \left( 1+\frac 1{p^{\varepsilon}} \right)   \\ &  \hspace{2cm} \times\prod_{p\mid N_E} \bigg(1- \frac{1}{p^{\varepsilon}} \bigg)^{-1}\left( 1+\frac 1{p^{2\varepsilon}} \right) \prod_{p\mid n} \left( 1-\frac 1{p^{2\varepsilon}} \right)^{-1} (|t|+1)^{-2} X^{\varepsilon} dt  \\
& \ll_{\varepsilon}  (|n|XN_E)^{\varepsilon}.
\end{align*}
The proof is finished by combining this estimate with \eqref{eq:seperation in two sums} and by noting that $ \mathcal M w(1) = \widehat w(0)/2 . $
\end{proof}

The next lemma is used to understand the first main term in the $1$-level density (see for example \eqref{equation ratios conj} or \eqref{equation main thm}).

\begin{lemma}
\label{lemma logd}
Fix $\varepsilon>0$, and assume the Riemann Hypothesis (RH). We have the estimate\footnote{Remark \ref{remarkgeneralinteger} also applies here.}
\begin{multline*} \frac 1{W(X)} \sumt  \log |d|= \log X + \frac 2{\widehat w(0)}\int_0^{\infty} w(x) \log x  dx +\sum_{p\mid N_E} \frac{2\log p}{p^2-1} +2\frac{\zeta'(2)}{\zeta(2)}\\ - \prod_{p\mid N_E} \left( 1-\frac 1{p^{1/2}} \right)\left( 1-\frac 1{p} \right)^{-1}\frac{\mathcal Mw(\tfrac 12)}{\mathcal M w(1)}\zeta(\tfrac12)X^{-1/2} +O_{\varepsilon,w}\big(N_E^{\varepsilon}X^{-3/4 +\varepsilon}\big).  
\end{multline*}
\end{lemma}

\begin{proof}
The proof follows closely that of Lemma \ref{lemmaweighted}. One writes 
\begin{equation}
\sumt  \log |d| = -\frac 2{2\pi i}\int_{\Re(s) = 2} \bigg( \sum_{ \substack{d\geq 1 \\ (d,N_E)=1}} \frac{\mu^2(d)}{d^s}  \bigg)' \prod_{p\mid N_E} \left(1- \frac 1{p^{2s}} \right) \zeta(2s)  \mathcal Mw(s)X^sds,
\label{equation log d proof}
\end{equation} 
and pulls the contour of integration to the left until the line $\Re(s)=1/4+\varepsilon$. Note that under RH, the only poles of the function 
\begin{align*}Z(s):&= \bigg( \sum_{ \substack{d\geq 1 \\ (d,N_E)=1}} \frac{\mu^2(d)}{d^s}  \bigg)' \prod_{p\mid N_E} \left(1- \frac 1{p^{2s}} \right) \zeta(2s) \\&= \prod_{p\mid N_E} \left( 1-\frac 1{p^s}\right) \bigg[ \zeta'(s) -\frac{2\zeta(s)\zeta'(2s)}{\zeta(2s)}+\zeta(s)\sum_{p\mid N_E} \frac{\log p}{p^s+1}\bigg] 
\end{align*}
in the region $1/4<\Re(s)\leq 2$ are at $s=1$ and at $s=1/2$. The value of the residues of the integrand in \eqref{equation log d proof} are obtained from a straightforward computation.
\end{proof}
\begin{remark}
One can pull the contour of integration further to the left in \eqref{equation log d proof}, and obtain an unconditional estimate with an error term of size $O_{\varepsilon,E}(X^{\varepsilon})$. This estimate will contain terms of the form 
$$2\zeta(\rho/2) \mathcal M w(\rho/2) X^{\rho/2}\prod_{p\mid N_E} \left( 1-\frac 1{p^{\rho/2}}\right) ,$$
with $\rho$ running over the nontrivial zeros of $\zeta(s)$.
\end{remark}

We also prove a version of Lemma \ref{lemmaweighted} which will be important in the analysis of $\mathcal D^*(\phi;X)$.

\begin{lemma}
\label{lemma count of squarefree}
Fix $n\in \mathbb N$ and $\varepsilon>0$. Under the Riemann Hypothesis (RH), we have the estimate\footnote{Remark \ref{remarkgeneralinteger} also applies here.}

\begin{multline*}
\sumn \left(\frac dn\right)=  \kappa(n) \frac{X}{\zeta(2)} \widehat w(0)   \prod_{p\mid n}  \left(1+\frac 1{p} \right)^{-1} \prod_{p\mid N_E} \bigg(1+ \frac{1}{p} \bigg)^{-1} \prod_{p\mid (n,N_E)} \bigg(1+ \frac{1}{p} \bigg)\\+O_{\varepsilon,w}\big((N_E)^{\varepsilon}|n|^{\frac 38(1-\kappa(n))+\varepsilon}X^{\frac 14+\varepsilon}\big),
\end{multline*}
where
$$ \kappa(n) := \begin{cases}
1 &\text{ if } n = \square, \\
0 &\text{ otherwise.}
\end{cases}$$
\end{lemma}

\begin{proof}
Since $w(t)$ is even, we have
\begin{equation}
\sumn \left( \frac dn\right) = \left(1+\left( \frac{-1}n \right)\right)\underset{\substack{ d\geq 1 \\ (d,N_E)=1}}{{\sum}^*}  \wn \left( \frac dn\right).
\end{equation}
An application of Mellin inversion and a straightforward calculation shows that
\begin{align*}
J_n(X):=\sum_{ \substack{d\geq 1 \\ (d,N_E)=1}} \mu^2(d) \wn \left( \frac dn\right) =  \frac 1{2\pi i } \int_{\Re(s)=2} Z_{n,E}(s)\mathcal Mw(s) X^s ds,
\end{align*}
where $$ Z_{n,E}(s):= \frac{L\big(s,\left(\frac{\cdot}n \right)\big)}{\zeta(2s)} \prod_{p\mid N_E} \bigg(1+ \frac{\left(\frac pn \right)}{p^s} \bigg)^{-1} \prod_{p\mid n} \left( 1-\frac 1{p^{2s}} \right)^{-1}. $$

In the case when $n$ is not a square, $L(s,\left(\frac{\cdot}n \right))$ is holomorphic at $s=1$. We can thus shift the contour of integration to the left until the line $\Re(s)=\frac 14+\varepsilon$, since by the Riemann Hypothesis, the zeros of  $\zeta(2s)$ all have real part at most $\frac 14$. We apply the estimates \cite[(5.20)]{IK}, \cite[Thm.\ 13.23]{MV} and Lemma \ref{lemma fast decay} together with the rapid decay of $\mathcal Mw(s)$ on vertical lines (rather than following the proof of \cite[Thm.\ 13.24]{MV} directly), to obtain, for some $C>0$,
\begin{align*}
 J_n(X) &= \frac 1{2\pi i} \left( \int_{\substack{\Re(s)=\frac 14+\varepsilon \\ |\Im(s)|\leq 5 }}+ \int_{\substack{\Re(s)=\frac 14+\varepsilon \\ |\Im(s)|> 5}} \right) Z_{n,E}(s)\mathcal Mw(s) X^s ds
\\ & \ll_{\varepsilon,C,w} (N_E)^{\varepsilon} |n|^{\frac 38+\varepsilon}X^{\frac 14+\varepsilon}+(|n|N_E)^{\frac{\varepsilon}2}\int_{ |t|>5 } (|t|+4)^{\frac {C \log ((2\varepsilon)^{-1})}{\log\log (|t|+4)}}\frac{(|n|(|t|+1))^{\frac 38+\frac{\varepsilon}2}}{(|t|+1)^{2C \log((2\varepsilon)^{-1})+2}}   X^{\frac 14+\varepsilon}dt  \\&\ll_{\varepsilon,C} (N_E)^{\varepsilon} |n|^{\frac 38+\varepsilon}X^{\frac 14+\varepsilon} . 
\end{align*}

In the case where $n$ is a square, the proof is similar, except that the integral we are interested in is given by 
$$ K_n(X):=\frac 1{2\pi i } \int_{\Re(s)=2} \frac{\zeta(s)}{\zeta(2s)} \prod_{p\mid N_E} \bigg(1+ \frac{\left(\frac pn \right)}{p^s} \bigg)^{-1} \prod_{p\mid n} \left(1+\frac 1{p^s} \right)^{-1}\mathcal Mw(s) X^s ds.$$
Shifting the contour of integration to the left until $\Re(s)=\frac 14+\varepsilon$, we pick up the residue from the simple pole at $s=1$ and arrive at the formula
$$ K_n(X)= \frac {\mathcal M w(1) X}{\zeta(2)} \prod_{p\mid N_E} \bigg(1+ \frac{\left(\frac pn \right)}{p} \bigg)^{-1} \prod_{p\mid n}  \left(1+\frac 1{p} \right)^{-1} + O_{\varepsilon,w}\big((|n|N_E)^{\varepsilon} X^{\frac 14+\varepsilon}\big). $$
The error term comes from the same reasoning as before, except that we used the convexity bound for $\zeta(s)$ instead of that of $L\big(s,\left(\frac{\cdot}n \right)\big)$.
\end{proof}

\subsection{The Explicit Formula}

The fundamental tool to study the $1$-level density is the explicit formula; we will use Mestre's version \cite{Me}. Recall that $L=\log(N_E X^2/(2\pi e)^2)$.
\begin{lemma}[Explicit Formula]
Let $\phi$ denote a Schwartz function whose Fourier transform has compact support.  For $d$ square-free with $(d,N_E)=1$, we have the formula
\begin{multline} D_X(E_d;\phi)= \frac{\widehat \phi(0)}L \log \bigg(\frac{N_{E}d^2}{(2\pi)^2}\bigg)-\frac{2} L\int_0^{\infty} \left( \frac{\widehat \phi(x/L) e^{-x} }{1-e^{-x}} - \widehat \phi(0) \frac{e^{-x}}x \right)dx\\ -\frac 2L\sum_{p,m} \frac{(\alpha_E(p)^m+\beta_E(p)^m) \chi_d (p^m) \log p}{p^{m/2}} \widehat \phi\left( \frac{m \log p}{L} \right).
\label{equation explcit formula}
\end{multline}
\end{lemma}

\begin{proof}
Recall that
$$ D_X(E_d;\phi) = \sum_{\gamma_d} \phi\left( \gamma_d \frac{L}{2\pi} \right), $$
where $\rho_d=\frac 12+i\gamma_d$ runs through the non-trivial zeros of $L(s,f_E\otimes \chi_d)$. Note that since $(d,N_E)=1$ and $d$ is square-free, Proposition 14.20 of \cite{IK} implies that $f_E \otimes \chi_d$ is a weight $2$ newform of level~$N_Ed^2$.

We take $\Phi(s):= \phi(\frac{-iL}{2\pi} (s-1/2))$ in the explicit formula on page 215 of \cite{Me}, which applies to any weight $2$ newform on $\Gamma_0(N_Ed^2)$. This yields the formula
\begin{multline}  D_X(E_d;\phi) =F(0) \big(\log (N_{E}d^2) - 2\log(2\pi)\big)\\-2 \int_0^{\infty} \left( \frac{F(x) e^{-x} }{1-e^{-x}} -  F(0) \frac{e^{-x}}x \right)dx -2\sum_{p,m} \frac{(\alpha_E(p)^m+\beta_E(p)^m) \chi_d(p^m) \log p}{p^{m/2}} F(m \log p),
\label{equation explicit 1}
\end{multline}
where the function $F(x)$ is such that
$$ \Phi(s) = \int_{\mathbb R} F(x) e^{x(s-1/2)} dx. $$
In \eqref{equation explicit 1} we have used the identity
$$\alpha_{f_E\otimes \chi_d}(p)^m+\beta_{f_E\otimes \chi_d}(p)^m = (\alpha_E(p)^m+\beta_E(p)^m) \chi_d(p^m). $$
To show this identity, note that since $f_E\otimes \chi_d$ is a newform, its $L$-function is given by
\begin{align*}
\sum_{n\geq 1} \frac{\lambda_{E}(n) \chi_d(n)}{n^s} &= \prod_p \left( 1-\frac{\alpha_{f_E\otimes \chi_d}(p)}{p^s} \right)^{-1} \left( 1-\frac{\beta_{f_E\otimes \chi_d}(p)}{p^s} \right)^{-1} \\ &= \prod_p \left( 1-\frac{\lambda_{E}(p)\chi_d(p)}{p^s} + \frac {\psi_{d^2N_E}(p)}{p^{2s}}\right)^{-1},
\end{align*}
where $\psi_k(n)$ is the principal character modulo $k$. Hence the choice $\alpha_{f_E\otimes \chi_d}(p) = \alpha_E(p) \chi_d(p) $ and $\beta_{f_E\otimes \chi_d}(p) = \beta_E(p)\chi_d(p)$ is consistent (note that $\chi_d(p)$ is real and the pair $(\alpha_{f_E\otimes \chi_d}(p),\beta_{f_E\otimes \chi_d}(p))$ is defined up to permutation).

We have that
$$ \phi(tL) = \Phi\left( \frac 12+2\pi it\right) = \int_{\mathbb R}   F(x) e^{2\pi it x } dx = \widehat F(-t); $$
hence $F(x)=\widehat \phi(x/L)/L$ gives us the desired choice of $\Phi(s)$, and \eqref{equation explcit formula} follows.
\end{proof}

\begin{corollary}
\label{corollary 1 level densities}
We have the following formulas for the $1$-level densities we are interested in (see \eqref{equation one level density} and \eqref{equation one level density star}):

\begin{multline} \mathcal D^*(\phi;X) =  \frac {\widehat \phi(0)}{ LW^*(X)} \sumn \log \bigg(\frac{N_{E}d^2}{(2\pi)^2}\bigg)-\frac{2} L\int_0^{\infty} \left( \frac{\widehat \phi(x/L) e^{-x} }{1-e^{-x}} - \widehat \phi(0) \frac{e^{-x}}x \right)dx \\ -\frac 2{L  W^*(X)  }\sum_{p,m} \frac{(\alpha_E(p)^m+\beta_E(p)^m) \log p}{p^{m/2}} \widehat \phi\left( \frac{m \log p}{L} \right) \sumn \left( \frac{d}{p^m}\right),
\label{equation one level density explicit formula squarefree}
\end{multline}

\begin{multline} \mathcal D(\phi;X) = \frac {\widehat \phi(0)}{LW(X)}  \sumt \log \bigg(\frac{N_{E}d^2}{(2\pi)^2}\bigg)  -\frac{2} L\int_0^{\infty} \left( \frac{\widehat \phi(x/L) e^{-x} }{1-e^{-x}} - \widehat \phi(0) \frac{e^{-x}}x \right)dx \\ -\frac 2{L W(X)  }\sum_{p,m} \frac{(\alpha_E(p)^m+\beta_E(p)^m) \log p}{p^{m/2}} \widehat \phi\left( \frac{m \log p}{L} \right) \sumt  \left( \frac{d}{p^m}\right).
\label{equation one level density explicit formula}
\end{multline}
(Recall the definition of $\widetilde w$ given in \eqref{wtildeintro}.)
\end{corollary}

\begin{proof}
The idea is to sum \eqref{equation explcit formula} over the desired values of $d$ required to obtain $\mathcal D^*(\phi;X)$ and $\mathcal D(\phi;X)$, against the smooth weight $w(\tfrac dX)$. The first identity follows immediately from \eqref{equation explcit formula}.
For the second identity, note that for any integer $\ell \geq 1$,
$$ D_X(E_d;\phi) = D_X(E_{\ell ^2 d};\phi). $$
This follows from the fact that $L(s,E_d)$ and $L(s,E_{\ell^2 d})$ have the same nontrivial zeros, since $(x,y) \mapsto (x,\ell y)$ induces a bijection between the groups $E_{\ell^2 d}(\mathbb F_p)$ and $E_d(\mathbb F_p)$ for any $p\nmid \ell dN_E$.
Hence,
\begin{align}\label{alternative1levelformula}
W(X) \mathcal D(\phi;X) = \sum_{(d,N_E)=1} w\left( \frac dX \right) D_X(E_d;\phi) &= \sumnn \sum_{\substack{n \geq 1 \\ (n,N_E)=1}} w\left( \frac {n^2d}X \right) D_X(E_{d};\phi) \nonumber \\
&= \sumt D_X(E_{d};\phi).
\end{align}
Finally, by applying \eqref{equation explcit formula} and noting that 
\begin{equation}\label{thedubsarethesame}
W(X)=\sum_{(d,N_E)=1} w\left ( \frac dX \right) =\sumt,
\end{equation}
we deduce \eqref{equation one level density explicit formula}. 
\end{proof}

\section{The prime sum in $\mathcal D(\phi;X)$}\label{PRIMESUMSEC}

\label{section all d}
The goal of this section is to study the second prime sum appearing in Corollary \ref{corollary 1 level densities}, 
that is the term
\begin{equation}  -\frac 2{LW(X)}\sum_{p,m} \frac{(\alpha_E(p)^m+\beta_E(p)^m) \log p}{p^{m/2}} \widehat
 \phi\left( \frac{m \log p}{L} \right) \sumt \left( \frac{d}{p^m}\right) = S_{\text{odd}} +S_{\text{even}},
 \label{equation seperation Sodd Seven}
\end{equation}
where $S_{\text{odd}}$ and $S_{\text{even}}$ contain respectively the terms with odd and even $m$. In Appendix \ref{Appendix A}, we will see that $S_{\text{even}}$ appears as is in the Ratios Conjecture's prediction (see Theorem \ref{extendedoneleveldensityresult}). Bounding $S_{\text{odd}}$ constitutes the heart of the paper, and sets the limit for both the allowable support for the test function $\phi$ as well as the size of the error term in Theorems \ref{main theorem} and \ref{second main theorem}.
Our analysis is inspired by that of Katz and Sarnak \cite{KS1}, who used Poisson summation to analyze such a quantity. This will be done in Lemma \ref{lemma:initialPoisson}, but we first show that the terms with odd $m\geq 3$ are negligible. In this section, we do not indicate the dependence on $\phi$ and $w$ of the implied constants in the error terms.

\begin{lemma}
\label{lemma S odd big m}
Fix $\varepsilon>0$. Assuming that $\sigma:= \text{sup}(\text{supp} \widehat \phi)< \infty, $ we have the bound
$$-\frac 2{LW(X)}\sum_{\substack{p \\ m \geq 3 \text{ odd}}} \frac{(\alpha_E(p)^m+\beta_E(p)^m) \log p}{p^{m/2}} \widehat \phi\left( \frac{m \log p}{L} \right) \sumt \left( \frac{d}{p^m}\right)\ll_{\varepsilon} N_E^{\varepsilon}X^{-1+\varepsilon}.$$
\end{lemma}

\begin{proof}
After noting that $\left( \frac{d}{p^m}\right)=\left( \frac{d}{p}\right)$, this is a direct application of Lemma \ref{lemmaweighted}, combined with the bounds $|\alpha_E(p)|, |\beta_E(p)| \leq 1$ and Remark \ref{whatiswx}.
\end{proof}

We now adapt the arguments of \cite{KS1}.

\begin{lemma}
Fix $\varepsilon >0$. Assuming that $\sigma:= \text{sup}(\text{supp} \widehat \phi)< \infty, $ we have the following:
\begin{multline*}
S_{\text{odd}}  =   -\frac {2X}{LW(X)} \sum_{0\leq k\leq  10 \log X} \sum_{\ell \mid N_E} \frac{\mu(\ell)}{\ell} \sum_{\substack{p \nmid 2 N_E }} \frac{ \left( \frac{-\ell}p\right)\overline{\epsilon_p} \lambda_E(p) \log p}{p^{1+2k}}\widehat \phi\left( \frac{\log p}{L} \right)   \sum_{t \in \mathbb Z}\left( \frac t{p} \right) \widehat w\left( \frac{Xt}{p^{1+2k}\ell}\right)\\+O_{\varepsilon}\left( N_E^{\varepsilon} X^{-1+\varepsilon}\right),
\end{multline*}
\label{lemma:initialPoisson}
where
$$\epsilon_p:= \begin{cases}
1 &\text{ if } p\equiv 1 \bmod 4, \\
i &\text{ if } p\equiv 3 \bmod 4.
\end{cases} $$
\end{lemma}

\begin{proof}

First note that the terms with $p\mid 2N_E$ in $S_{\text{odd}}$ are negligible, since by Lemma \ref{lemmaweighted} we have the bound
\begin{align*}
&\frac 2{LW(X)} \sum_{\substack{ p\mid 2N_E}} \frac{ (\alpha_E(p)+\beta_E(p)) \log p}{p^{1/2}} \widehat \phi\left( \frac{ \log p}{L} \right)  \sumt\left( \frac{d}{p}\right)  \\
&\ll_{\varepsilon} \frac 1X \sum_{p\mid 2N_E}  \log p (N_EXp)^{\varepsilon/3} \ll_{\varepsilon} N_E^{\epsilon}X^{-1+\varepsilon}.
\end{align*}
Note also that by the definition of $\widetilde w(x)$, we have for $p\nmid 2N_E$ the identity
$$  \sumt \left( \frac{d}{p}\right)= \sum_{k\geq 0} \sum_{(d,N_E)=1} w\left( \frac{d}{X/p^{2k}}\right)\left( \frac{d}{p}\right) .   $$
Recall that $\alpha_E(p)+\beta_E(p)=\lambda_E(p)$. Therefore, by Lemma \ref{lemma S odd big m} we have that
$$S_{\text{odd}} = -\frac 2{LW(X)} \sum_{0 \leq k\leq 10 \log X} \sum_{\substack{p \nmid 2N_E }} \frac{ \lambda_E(p) \log p}{p^{1/2}} \widehat \phi\left( \frac{ \log p}{L} \right)  \sum_{(d,N_E)=1} w\left( \frac d{X/p^{2k}}\right)\left( \frac{d}{p}\right) +O_{\varepsilon}(N_E^{\varepsilon}X^{-1+\varepsilon}). $$
In the last expression we have removed the terms with $k>10 \log X$, since by the rapid decay of $w(t)$, their sum is (we write $c_E:=N_E/(2\pi e)^2$)
$$\ll \frac{N_E^{\varepsilon}}{X} \sum_{ k> 10 \log X} \sum_{\substack{p \leq (c_EX^2)^{\sigma} }} \frac{ \log p}{p^{1/2}} \sum_{(d,N_E)=1} \left( \frac X{dp^{2k}}\right)^{2} \ll \frac {N_E^{\varepsilon}X}{2^{30 \log X}}  \sum_{\substack{p }} \frac{ \log p}{p^{3/2}} \ll N_E^{\varepsilon}X^{-1}.  $$

We now introduce additive characters using Gauss sums. These characters have the advantage of being smooth functions of their argument and will thus allow us to use Poisson summation. For $p$ an odd prime and $Y>0$ we write (see \cite[Sections 2 and 9]{D} for the definition and properties of the Gauss sum $\tau(\chi)$)
\begin{align} \label{first transformation} \sum_{(d,N_E)=1 }  w\left(\frac dY \right)\left( \frac{d}{p}\right)&= \sum_{\ell \mid N_E} \mu(\ell) \sum_{r \in \mathbb Z} w\left(\frac {r \ell}Y \right)\left( \frac{r \ell}{p}\right) \\&=  \sum_{\ell \mid N_E} \mu(\ell) \sum_{r \in \mathbb Z}  w\left(\frac{r \ell}Y \right) \frac{1}{\tau\left( \left( \frac{\cdot}{p}\right)\right)} \sum_{b\bmod p} \left( \frac bp\right) e\left(\frac{r \ell b}p \right) \notag \\&= \sum_{\ell \mid N_E} \mu(\ell) \frac{\overline{\epsilon_p}}{p^{\frac 12}}  \sum_{b\bmod p} \left( \frac {b}p\right) \sum_{r \in \mathbb Z} w\left(\frac{r \ell}Y \right)  e\left(\frac{r \ell b}p \right). \notag
\end{align}
Our expression for $S_{\text{odd}}$ is now
\begin{multline}
S_{\text{odd}} = -\frac{2}{LW(X)}\sum_{0\leq k\leq  10 \log X} \sum_{\ell \mid N_E} \mu(\ell)  \Bigg[\sum_{\substack{ p \nmid 2\ell N_E}} \frac{\overline{\epsilon_p} \lambda_E(p) \log p}{p}\widehat \phi\left( \frac{ \log p}{L} \right)  \\  \times \sum_{b\bmod p} \left( \frac bp\right)  \sum_{r \in \mathbb Z}  w\left(   \frac {r\ell}{X/p^{2k}}\right) e\left(\frac{r \ell b}p \right)\Bigg]
+O_{\varepsilon}(N_E^{\varepsilon}X^{-1+\varepsilon}).
\label{equation Sodd to be poissonized}
\end{multline}
Notice that we removed the terms with $p\mid \ell$ since they are all zero. This can be seen from the last expression using the orthogonality of $\left( \frac {\cdot} p\right)$, and is even more apparent in \eqref{first transformation}. We are ready to apply Poisson summation in \eqref{equation Sodd to be poissonized}:
$$  \sum_{r \in \mathbb Z}  w\left(\frac{r\ell}Y \right) e\left(\frac{r \ell b}p \right) =\frac Y{\ell} \sum_{s \in \mathbb Z} \widehat w\left( Y\left( \frac{s}{\ell} - \frac {b}p\right)\right), $$
which yields the expression
\begin{multline*}
S_{\text{odd}} = -\frac{2X}{LW(X)} \sum_{0\leq k\leq  10 \log X}\sum_{\ell \mid N_E} \frac{\mu(\ell)}{\ell}  \Bigg[\sum_{\substack{ p\nmid 2\ell N_E}} \frac{\overline{\epsilon_p} \lambda_E(p) \log p}{p^{1+2k}}\widehat \phi\left( \frac{ \log p}{L} \right) \\  \times \sum_{b\bmod p} \left( \frac bp\right)  \sum_{s \in \mathbb Z} \widehat w\left(  \frac{X}{p^{2k}}\left( \frac s{\ell}- \frac {b}p\right)\right)\Bigg]
+O_{\varepsilon}(N_E^{\varepsilon}X^{-1+\varepsilon}).
\end{multline*}
Note that as $s$ runs through the integers and $b$ runs through a complete residue system modulo $p$, the variable $t:=sp-b\ell$ runs through all integers (the fact that $(\ell,p)=1$ is crucial here). In other words, the following map is a group isomorphism:
\begin{align*}
f_{\ell,p} : \mathbb Z/ p\mathbb Z \times \mathbb Z &\longrightarrow \mathbb Z \\
 (b,s) &\longmapsto sp-b\ell.
\end{align*}
Combining this with the fact that $\left( \frac{b}{p}\right) = \left( \frac{-\ell^{-1}(t-sp)}{p}\right) = \left( \frac{-\ell}{p}\right) \left( \frac{t}{p}\right) $, we obtain
\begin{multline*}
S_{\text{odd}} = -\frac{2X}{LW(X)} \sum_{0\leq k\leq  10 \log X} \sum_{\ell \mid N_E} \frac{\mu(\ell)}{\ell}  \Bigg[\sum_{\substack{  p\nmid 2\ell N_E}} \frac{ \left(\frac{-\ell}p \right)\overline{\epsilon_p} \lambda_E(p) \log p}{p^{1+2k}}\widehat \phi\left( \frac{ \log p}{L} \right) \\ \times \sum_{t\in\mathbb Z} \left( \frac tp\right)  \widehat w\left(  \frac {Xt}{ p^{1+2k}\ell} \right)\Bigg]
+O_{\varepsilon}(N_E^{\varepsilon}X^{-1+\varepsilon}).
\end{multline*}
\end{proof}

\begin{lemma}
\label{lemma:smallp}
Fix $\varepsilon>0$. We have the bound
$$S_1:=\sum_{0\leq k\leq  10 \log X} \sum_{\ell \mid N_E} \frac{\mu(\ell)}{\ell} \sum_{\substack{p \nmid 2 N_E \\ p\leq X^{ \frac{1-\varepsilon}{2k+1} } }} \frac{ \left( \frac{-\ell}p\right)\overline{\epsilon_p} \lambda_E(p) \log p}{p^{1+2k}}\widehat \phi\left( \frac{\log p}{L} \right)   \sum_{t \in \mathbb Z}\left( \frac t{p} \right) \widehat w\left( \frac{Xt}{p^{1+2k}\ell}\right) \ll_{\varepsilon,E} X^{-1}. $$
\end{lemma}

\begin{proof}
Letting $M=1+\max(10,\varepsilon^{-1})$, we have by the rapid decay of $\widehat w$ that
$$ S_1 \ll_{\varepsilon} \sum_{0\leq k\leq  10 \log X} \sum_{\ell \mid N_E} \frac 1{\ell}  \sum_{\substack{p \nmid 2 N_E \\ p\leq X^{ \frac{1-\varepsilon}{2k+1} } }} \frac{ \log p}{p^{1+2k}} \sum_{ 0\neq t \in \mathbb Z} \left(\frac{\ell}{tX^{\varepsilon}}\right)^M \ll_{M,E} X^{-M\varepsilon} \log X. $$
\end{proof}

\begin{lemma}
\label{lemma:0.6}
Fix $K\in \mathbb N$ and $\varepsilon>0$. If $\sigma :=$sup$(\text{supp}\widehat \phi)<\infty$, then we have the bound
\begin{multline*}S_{2,K}:=\sum_{K\leq k\leq  10 \log X} \sum_{\ell \mid N_E} \frac{\mu(\ell)}{\ell} \sum_{\substack{p \nmid 2 N_E \\ p> X^{ \frac{1-\varepsilon}{2k+1} } }} \frac{ \left( \frac{-\ell}p\right)\overline{\epsilon_p} \lambda_E(p) \log p}{p^{1+2k}}\widehat \phi\left( \frac{\log p}{L} \right)   \sum_{t \in \mathbb Z}\left( \frac t{p} \right) \widehat w\left( \frac{Xt}{p^{1+2k}\ell}\right) \\ \ll_{\varepsilon,E} X^{-\max(\frac{4K-1}{4K+1},1-2\sigma)+\varepsilon}.
\end{multline*}
\end{lemma}
\begin{proof}

We split the sum over $p$ into two parts, cutting at the point $X^{\frac 1{2k+1/2}}$. To bound the first of these sums, we first note that
\begin{equation}
 \sum_{t \in \mathbb Z}\left( \frac t{p} \right) \widehat w\left( \frac{Xt}{p^{1+2k}\ell}\right) \ll \sum_{0<|t| \leq \frac{p^{1+2k}\ell}X} 1+ \sum_{|t| > \frac{p^{1+2k}\ell}X} \left(\frac{p^{1+2k}\ell}{Xt}\right)^2\ll_{E} \frac{p^{1+2k}}{X},
 \label{equation:boundsumfouriertrans}
\end{equation}
and hence, writing $c_E:=N_E/(2\pi e)^2$,
\begin{align}
\sum_{K\leq k\leq  10 \log X}& \sum_{\ell \mid N_E} \frac{\mu(\ell)}{\ell} \sum_{\substack{p \nmid 2 N_E \\  X^{ \frac{1-\varepsilon}{2k+1}} < p \leq X^{\frac 1{2k+1/2}}  }} \frac{ \left( \frac{-\ell}p\right)\overline{\epsilon_p} \lambda_E(p) \log p}{p^{1+2k}}\widehat \phi\left( \frac{\log p}{L} \right)   \sum_{t \in \mathbb Z}\left( \frac t{p} \right) \widehat w\left( \frac{Xt}{p^{1+2k}\ell}\right) \label{equation improvable with Riemann} \\
&\ll_E  \sum_{K\leq k\leq  10 \log X}  \sum_{\substack{ X^{ \frac{1-\varepsilon}{2k+1}} < p \leq  \min\big(X^{\frac 1{2k+1/2}},(c_EX^2)^{\sigma}\big)  }} \frac {\log p}{X} \ll_E \frac{\log X}{X^{\max\big(\frac{2K-1/2}{2K+1/2},1-2\sigma\big)}}. \notag
\end{align} 
For those $k$ for which $2\sigma <(2k+1/2)^{-1} $, we have already covered the whole range of values of $p$ (for $X \gg_E 1$). For the remaining values of $k$, we bound the rest of the terms $\big(p> X^{\frac 1{2k+1/2}}\big)$ using the P\'olya-Vinogradov inequality, which reads
$$S(T):= \sum_{1\leq u\leq T}  \left( \frac u{p} \right)  \ll p^{\frac 12} \log p.$$
We then have
\begin{align*}
 \sum_{t \geq 0}\left( \frac t{p} \right) \widehat w\left( \frac{Xt}{p^{1+2k}\ell}\right) &= \int_{0}^{\infty} \widehat w\left( \frac{Xt}{p^{1+2k}\ell}\right) dS(t)= -\int_0^{\infty} \frac{X}{p^{1+2k}\ell} \widehat w'\left( \frac{Xt}{p^{1+2k}\ell}\right)S(t) dt \\ &\ll p^{\frac 12} \log p.
 \end{align*}
Treating the terms with $t<0$ in a similar way, we conclude that the second part of $S_{2,K}$, that is the sum over $ K' \leq k \leq 10 \log X$ with $K' := \max(K,\frac 14 ( \frac 1{\sigma}-1))$, satisfies
\begin{align*}
\sum_{K'\leq k\leq  10 \log X}& \sum_{\ell \mid N_E} \frac{\mu(\ell)}{\ell} \sum_{\substack{p \nmid 2 N_E \\ p> X^{ \frac{1}{2k+1/2}}   }} \frac{ \left( \frac{-\ell}p\right)\overline{\epsilon_p} \lambda_E(p) \log p}{p^{1+2k}}\widehat \phi\left( \frac{\log p}{L} \right)   \sum_{t \in \mathbb Z}\left( \frac t{p} \right) \widehat w\left( \frac{Xt}{p^{1+2k}\ell}\right) \\
&\ll_E  \sum_{K'\leq k\leq  10 \log X}  \sum_{\substack{ p> X^{ \frac{1}{2k+1/2}}  }} \frac {(\log p)^2}{p^{2k+1/2}} \ll \frac{\log X}{X^{\frac{2K'-1/2}{2K'+1/2}}}= \frac{\log X}{X^{\max\big(\frac{2K-1/2}{2K+1/2},1-2\sigma\big)}}.
\end{align*}
This concludes the proof.
\end{proof}

\begin{remark}
In the proof of Lemma \ref{lemma:0.6}, we could have used Burgess's bound to improve our estimate on the sum with $p\in \big[X^{\frac 1{2k+3/4}},X^{\frac 1{2k+1/2}}\big]$. This would have yielded a better overall bound when $2\sigma \in (\tfrac 1{2k+3/4},\tfrac 1{2k+1/2}) $ for some $k\geq 1$. However, we have chosen to carry out this improvement in a separate paper \cite{FPS}. 
\end{remark}

\begin{proposition}
Fix $m\in \mathbb N$ and assume that $\frac 1{2(2m+1)}\leq \sigma =$sup$(\text{supp}\widehat \phi)< \frac 1{2(2m-1)}$. Then, for any fixed $\varepsilon>0$, we have the bound
$$ S_{\text{odd}} \ll_{\varepsilon,E}  X^{-\max(\frac {4m-1}{4m+1},1-2\sigma)+\varepsilon}.$$
\label{proposition:unconditionalbound}
\end{proposition}

\begin{proof}
By Lemma \ref{lemma:initialPoisson} and \ref{lemma:smallp}, we have that
\begin{multline*}
S_{\text{odd}}  =   -\frac {2X}{LW(X)} \sum_{0\leq k\leq  10 \log X} \sum_{\ell \mid N_E} \frac{\mu(\ell)}{\ell} \sum_{\substack{p \nmid 2 N_E \\ p> X^{\frac{1-\varepsilon}{2k+1}}}} \frac{ \left( \frac{-\ell}p\right)\overline{\epsilon_p} \lambda_E(p) \log p}{p^{1+2k}}\widehat \phi\left( \frac{\log p}{L} \right)  \\ \times \sum_{t \in \mathbb Z}\left( \frac t{p} \right) \widehat w\left( \frac{Xt}{p^{1+2k}\ell}\right)+O_{\varepsilon,E}\left(  X^{-1+\varepsilon}\right).
\end{multline*}
Note that for $X$ large enough in terms of $E$, the support of $\widehat \phi$ imposes the condition $p\leq X^{\frac 1{2m-1}-\eta}$ for some fixed $\eta>0$. Hence, all terms with $k\leq m-1$ in the above sum are identically zero. We then apply Lemma \ref{lemma:0.6} and obtain the bound
$$ S_{\text{odd}} \ll_{\varepsilon,E} X^{-\max(\frac {4m-1}{4m+1},1-2\sigma)+\varepsilon}.$$
\end{proof}

\begin{remark}
Notice that for $\sigma<\tfrac 12$, the error term $O_{\varepsilon,E}\big(X^{-\max(\frac {4m-1}{4m+1},1-2\sigma)+\varepsilon}\big)$ is always at most $O_{\varepsilon,E}(X^{-3/5+\varepsilon})$, which is sharper than the Ratios Conjecture's prediction. Moreover, if the support of $\widehat \phi$ is very small, then this error term is $O_{E}(X^{-1+\delta})$ with a very small $\delta$.
\end{remark}

In Proposition \ref{lemma:ECRHsmallsupport}, we give a sharper bound on $S_{\text{odd}}$, which is conditional on ECRH. We first give a standard application of ECRH.

\begin{lemma}
\label{lemma Riemann bound}
Assume ECRH. We have, for $m\in \mathbb Z_{\neq 0}$ and $y\geq 1$, the estimate
$$S_{m}(y):= \sum_{p\leq y} \left( \frac m p\right)  \lambda_E(p) \log p  \ll y^{\frac 12} (\log y)\log (2|m|yN_E). $$
\end{lemma}

\begin{proof}
The $L$-function
$$L(s,E_{m})= L(s,f_E\otimes \chi_{m}) = \sum_{n=1}^{\infty} \frac{\lambda_E(n) \left( \frac{m}n\right) }{n^s}  $$
is modular, and hence it admits an analytic continuation to the whole of $\mathbb C$ and has an Euler product and a functional equation. It is therefore an $L$-function in the sense of Iwaniec and Kowalski, and thus \cite[Thm.\ 5.15]{IK} takes the form
$$ \sum_{\substack{p^e\leq y \\ e\geq 1}} \big(\alpha_{E}(p)^e+\beta_{E}(p)^e\big)\log p \left( \frac{m}{p^e}\right) \ll  y^{\frac 12} (\log y) \log (2|m|^2N_Ey^2).  $$
The result follows by trivially bounding the contribution of prime powers.
\end{proof}

\begin{lemma}
\label{lemma:ECRHsmallsupport}
Fix $\varepsilon>0$ and $K\in \mathbb Z_{\geq 0}$, and assume ECRH. If $\sigma =$sup$(\text{supp}\widehat \phi)<\infty$, then we have the bound\footnote{In the case $K=0$, we adopt the convention that $\min(-1+\tfrac 1{4K},-1+\sigma) = -1+\sigma.$}
\begin{multline*}S_{2,K}:=\sum_{K\leq k\leq  10 \log X} \sum_{\ell \mid N_E} \frac{\mu(\ell)}{\ell} \sum_{\substack{p \nmid 2 N_E \\ p> X^{ \frac{1-\varepsilon}{2k+1} } }} \frac{ \left( \frac{-\ell}p\right)\overline{\epsilon_p} \lambda_E(p) \log p}{p^{1+2k}}\widehat \phi\left( \frac{\log p}{L} \right)   \sum_{t \in \mathbb Z}\left( \frac t{p} \right) \widehat w\left( \frac{Xt}{p^{1+2k}\ell}\right) \\ \ll_{\varepsilon,E} X^{\min(-1+\frac 1{4K},-1+\sigma)+\varepsilon}.
\end{multline*}
\end{lemma}

\begin{proof}
We will show that for $0\leq k \leq 10 \log X $, we have
\begin{equation*}
R:= \sum_{\ell \mid N_E} \frac{\mu(\ell)}{\ell} \sum_{t \in \mathbb Z} \sum_{\substack{p > X^{ \frac{1-\varepsilon}{2k+1}} }} \frac{ \left( \frac{-t \ell}p\right)\overline{\epsilon_p} \lambda_E(p) \log p}{p^{1+2k}}\widehat \phi\left( \frac{\log p}{L} \right)   \widehat w\left( \frac{Xt}{p^{1+2k}\ell}\right) \ll_{\varepsilon,E} X^{\min(-1+\frac 1{4k},-1+\sigma)+\varepsilon},
\end{equation*}
from which the lemma clearly follows. Notice that we have added back the primes dividing $2N_E$, since by a calculation similar to \eqref{equation:boundsumfouriertrans}, their contribution is
$$ \ll \sum_{\ell \mid N_E} \frac{1}{\ell}\sum_{t \in \mathbb Z} \sum_{\substack{p \mid 2N_E }} \frac{\log p}{p^{1+2k}} \left| \widehat w\left( \frac{Xt}{p^{1+2k}\ell}\right) \right| \ll \sum_{\ell \mid N_E} \frac{1}{\ell} \sum_{\substack{p \mid 2N_E }} \frac{\log p}{p^{1+2k}} \frac{p^{1+2k} \ell}X \ll_E X^{-1}.   $$
We now apply Lemma \ref{lemma Riemann bound}. Note that $\epsilon_p=\tfrac{1+i}2 \chi_0(p)+\tfrac{1-i}2 \chi_1(p)$, where $\chi_0$ and $\chi_1$ are respectively the trivial and the nontrivial character modulo $4$. Using this fact and applying Lemma \ref{lemma Riemann bound}, we have
$$ \mathcal S_{t,\ell}(y):= \sum_{p\leq y} \left( \frac {-t\ell} p\right) \overline{\epsilon_p} \lambda_E(p) \log p \ll y^{\frac 12} (\log y)\log(2|t|\ell yN_E). $$ 
Let $m(X):=\min(X^{\frac 1{2k}},(c_EX^2)^{\sigma})$, with $c_E:=N_E/(2\pi e)^2$. We first treat the terms in $R$ for which $X^{\frac{1-\varepsilon}{2k+1}}<p \leq m(X)$. Denoting the sum of these terms by $R_1$, we have 
$$R_1=  \sum_{\ell \mid N_E} \frac{\mu(\ell)}{\ell} \sum_{0\neq t \in \mathbb Z} P_{t,\ell},$$
where
$$ P_{t,\ell}:= \int_{X^{\frac{1-\varepsilon}{2k+1}}}^{m(X)} \frac{ \widehat \phi\left( \frac{\log x}L\right) \widehat w\left( \frac{Xt}{{x^{1+2k}}\ell}\right)}{x^{1+2k}} d\mathcal S_{t,\ell}(x). $$
Performing integration by parts, we obtain the bound
\begin{multline}
 P_{t,\ell} \ll \frac{ \big| \widehat w\left( \frac{X^{\varepsilon}t}{\ell} \right)\big| }{X^{1-\varepsilon}}\left|\mathcal S_{t,\ell}\left(X^{\frac{1-\varepsilon}{2k+1}}\right)\right| + \frac{ \big| \widehat w\big( \frac{Xt}{ m(X)^{1+2k}\ell} \big)\big| }{m(X)^{ 1+2k}}\left|\mathcal S_{t,\ell}\left(m(X)\right)\right|
 \\+  (k+1)\int_{X^{\frac{1-\varepsilon}{2k+1}}}^{m(X)}   \left[   \frac{ \big| \widehat w\left( \frac{Xt}{x^{1+2k}\ell} \right)\big| }{x^{2+2k}} + \frac{ \big|\widehat w'\left( \frac{Xt}{{x^{1+2k}}\ell} \right)\big| X|t|}{x^{3+4k}\ell} \right]x^{\frac 12} \big(\log (2x|t|\ell N_E)\big)^2dx.
\end{multline}
For any fixed $M\geq 1$ and $x\in [X^{\frac{1-\varepsilon}{2k+1}},m(X)]$, we have
$$ \sum_{0\neq t \in \mathbb Z}\left| \widehat w\left( \frac{X^{\varepsilon}t}{\ell} \right)\right| \log (2|t|) \ll_M \ell^M X^{-\varepsilon M} ; \hspace{.0cm}\sum_{0\neq t \in \mathbb Z}\left| \widehat w\left( \frac{Xt}{x^{1+2k}\ell} \right)\right| \big(\log (2|t|)\big)^2 \ll \frac{x^{1+2k} \ell}{X} \big(\log(2x^{1+2k}X\ell)\big)^2; $$
$$     \sum_{0\neq t \in \mathbb Z} X|t|\left|\widehat w'\left( \frac{Xt}{x^{1+2k}\ell} \right)\right| \big(\log (2|t|)\big)^2 \ll \frac{(x^{1+2k}\ell)^2}X\big(\log(2x^{1+2k}X\ell)\big)^2, $$
from which we obtain
\begin{align*}
R_1&\ll_{\varepsilon,E} X^{-1}+X^{\min(-1+\frac 1{4k},-1+\sigma)+\varepsilon}+ (k+1) \sum_{\ell \mid N_E} \frac{1}{\ell} \int_{X^{\frac{1-\varepsilon}{2k+1}}}^{m(X)}   \left[   \frac{ \ell }{x X} + \frac{ \ell }{x X} \right]x^{\frac 12} \big( \log(2x^{1+2k}X\ell)\log (2x\ell)\big)^2dx \\ &\ll_{\varepsilon,E} X^{\min(-1+\frac 1{4k},-1+\sigma)+\varepsilon},
\end{align*} 
since $k\leq 10 \log X$.

For those $k$ for which $2\sigma < \tfrac 1{2k}$, we have already covered all possible values of $p$ (for $X \gg_E 1$). For the remaining values of $k$, we apply the P\'olya-Vinogradov inequality in the exact same manner as in the proof of Lemma \ref{lemma:0.6}. Thus the sum of the terms with $p>X^{\frac 1{2k}}$ is 
$$  \ll_{\varepsilon} X^{-1+\frac 1{4k}+\varepsilon}. $$
The proof is complete since in this case we have that $\min(-1+\tfrac 1{4k},-1+\sigma)=-1+\tfrac 1{4k}$.
\end{proof}

\begin{proposition}
Assume ECRH, fix $m\in \mathbb N$ and assume that $\frac 1{2(2m+1)}\leq \sigma =$sup$(\text{supp}\widehat \phi)< \frac 1{2(2m-1)}$. Then, for any fixed $\varepsilon>0$, we have the bound
$$ S_{\text{odd}} \ll_{\varepsilon,E}  X^{\min(-1+\frac 1{4m},-1+\sigma)+\varepsilon}.$$
Moreover, for $\tfrac 12 \leq \sigma < 1$, we have 
$ S_{\text{odd}} \ll_{\varepsilon,E}  X^{-1+\sigma+\varepsilon}.$
\label{proposition:conditionalbound}
\end{proposition}
\begin{proof}
The proof is similar to that of Proposition \ref{proposition:unconditionalbound}, except that we substitute Lemma \ref{lemma:0.6} with Lemma \ref{lemma:ECRHsmallsupport}.
\end{proof}

We summarize the findings of this section in the following theorem.

\begin{theorem}
\label{theorem Sodd bound}
Fix $\varepsilon>0$. Then, in the range $\sigma=$sup$(\text{supp}\widehat \phi)<\frac 12$, we have the following unconditional bound:
$$ S_{\text{odd}} \ll_{\varepsilon,E} X^{\eta(\sigma)+\varepsilon}, $$
where $\eta(\sigma) =-\max(\frac {4m-1}{4m+1},1-2\sigma)$ for $\frac 1{2(2m+1)} \leq \sigma < \frac 1{2(2m-1)},$ with $m \in \mathbb N$.\footnote{Note that the domain of this function is $(0,\tfrac 12)$.}
Moreover, if we assume ECRH, then, in the wider range $\sigma=$sup$(\text{supp}\widehat \phi)<1$, we have the improved bound
$$ S_{\text{odd}} \ll_{\varepsilon,E} X^{\theta(\sigma)+\varepsilon}, $$
where
$$\theta(\sigma) = \begin{cases}
-\max(1-\frac {1}{4m},1-\sigma) &\text{ for } \frac 1{2(2m+1)} \leq \sigma < \frac 1{2(2m-1)}  \text{ with } m \in \mathbb N, \\
-1+\sigma &\text{ for } \frac 12\leq \sigma < 1.
\end{cases}$$
\end{theorem}

\begin{proof}
The unconditional bound follows directly from Proposition \ref{proposition:unconditionalbound}, and the conditional bound follows from Proposition \ref{proposition:conditionalbound}.
\end{proof}

We are now ready to complete the proof of our main result.

\begin{proof}[Proof of Theorems \ref{main theorem} and \ref{second main theorem}]
By Corollary \ref{corollary 1 level densities} and \eqref{equation seperation Sodd Seven}, we have
$$ \mathcal D(\phi;X) = \frac {\widehat \phi(0)}{LW(X)}  \sumt \log \bigg(\frac{N_{E}d^2}{(2\pi)^2}\bigg)  -\frac{2} L\int_0^{\infty} \left( \frac{\widehat \phi(x/L) e^{-x} }{1-e^{-x}} - \widehat \phi(0) \frac{e^{-x}}x \right)dx \\+ S_{\text{odd}} +S_{\text{even}}.  $$
Moreover, Theorem \ref{theorem Sodd bound} shows that we have $S_{\text{odd}} \ll_{\varepsilon} X^{\eta(\sigma)+\varepsilon}$ unconditionally, and that under ECRH we have $S_{\text{odd}} \ll_{\varepsilon} X^{\theta(\sigma)+\varepsilon}$. We conclude the proof by applying Lemma \ref{lemmaweighted} and Remark \ref{whatiswx} to $S_{\text{even}}$, yielding the estimate
$$ S_{\text{even}}=-\frac {2}{L}\sum_{\ell=1}^\infty \sum_{p} \frac{\left(\alpha_E(p)^{2\ell}+\beta_E(p)^{2\ell}\right) \log p}{p^{\ell}} \widehat \phi\left( \frac{2\ell\log p}{L} \right) \left( 1+\frac{\psi_{N_E}(p)}{p}\right)^{-1}+O_{\varepsilon}\left(X^{-1+\varepsilon}N_E^{\varepsilon}\right).$$
\end{proof}

\section{The prime sum in $\mathcal D^*(\phi;X)$}
\label{section square free}

In this section, we study the prime sum appearing in \eqref{equation one level density explicit formula squarefree}, that is
$$ -\frac 2{L  W^*(X)  }\sum_{p,m} \frac{(\alpha_E(p)^m+\beta_E(p)^m) \log p}{p^{m/2}} \widehat \phi\left( \frac{m \log p}{L} \right) \sumn \left( \frac{d}{p^m}\right) = S^*_{\text{odd}}+S^*_{\text{even}}, $$
where again $S^*_{\text{odd}}$ and $S^*_{\text{even}}$ denote respectively the sum of the terms with $m$ odd and even. Throughout, we do not indicate the dependence on $\phi$ and $w$ of the implied constants in the error terms.

We first give an estimate for $S^*_{\text{odd}}$, showing that the terms with $m\geq 3$ are negligible.

\begin{lemma}
\label{lemma:firststepsf}
Fix $\varepsilon>0$, and assume RH. Denoting by $[a,b]$ the least common multiple of $a$ and $b$, we have
\begin{multline*}
S^*_{\text{odd}} = -\frac 2{L  W^*(X)  }\sum_{\substack{ \ell \mid N_E \\ s \in \mathbb N}}  \mu(\ell)\mu(s)\sum_{p\nmid 2sN_E } \frac{\lambda_E(p) \log p}{p^{1/2}} \widehat \phi\left( \frac{ \log p}{L} \right) \sum_{0\neq d \in \mathbb Z} w\left( \frac{d[\ell,s^2]}{X}\right)\left( \frac{d[\ell,s^2]}{p}\right)\\ +O_{\varepsilon}\big(N_E^{\varepsilon}X^{-\frac 34+\varepsilon}\big). 
\end{multline*}
\end{lemma}

\begin{proof}
We first see that $\left( \frac d{p^m}\right)=\left( \frac d{p}\right)$, and so Lemma \ref{lemma count of squarefree} implies that 
$$ \frac {-2}{L  W^*(X)  } \sum_{\substack{p \\ m\geq 3 \text{ odd}}} \frac{(\alpha_E(p)^m+\beta_E(p)^m) \log p}{p^{m/2}} \widehat \phi\left( \frac{m \log p}{L} \right) \sumn \left( \frac{d}{p^m}\right)\ll_{\varepsilon} N_E^{\varepsilon}X^{-\frac 34+\varepsilon}. $$
The same lemma also implies the bound
$$-\frac 2{L  W^*(X)  }\sum_{p \mid 2N_E} \frac{\lambda_E(p) \log p}{p^{1/2}} \widehat \phi\left( \frac{ \log p}{L} \right) \sumn \left( \frac{d}{p}\right)\ll_{\varepsilon} N_E^{\varepsilon}X^{-\frac 34+\varepsilon}.$$
The claimed formula follows from using the identity $\mu^2(d)=\sum_{s^2 \mid d} \mu(s)$ and interchanging the order of summation.
\end{proof}

We now follow the arguments of \cite{KS1}.

\begin{lemma}
Fix $\varepsilon>0$. Assume RH and ECRH, and suppose that $\sigma:= \text{sup}(\text{supp} \widehat \phi)< \infty $. Then, for any $S\geq 1$, we have
\begin{multline*}
S^*_{\text{odd}}  =   \frac {-2X}{LW^*(X)}  \sum_{\substack{\ell \mid N_E \\ s \leq S}} \frac{\mu(\ell)\mu(s)}{[\ell, s^2]} \sum_{\substack{p \nmid 2 sN_E }} \frac{ \left( \frac{-[\ell,s^2] }p\right)\overline{\epsilon_p} \lambda_E(p) \log p}{p}\widehat \phi\left( \frac{\log p}{L} \right)   \sum_{t \in \mathbb Z}\left( \frac t{p} \right) \widehat w\left( \frac{Xt}{p[\ell, s^2]}\right)\\+O_{\varepsilon}\big(N_E^{\varepsilon} X^{\varepsilon}(\log S)^3S^{-1}+ N_E^{\varepsilon} X^{-\frac 34+\varepsilon}\big).
\end{multline*}
\label{lemma:secondPoisson}
\end{lemma}

\begin{proof}
The starting point is Lemma \ref{lemma:firststepsf}, in which we will bound the terms with $s> S$ using ECRH. Applying Lemma \ref{lemma Riemann bound} and a routine summation by parts, we obtain that
\begin{align*}
&\frac 2{L  W^*(X)  }\sum_{\substack{ \ell \mid N_E \\ s >S}}  \mu(\ell)\mu(s)\sum_{0\neq d \in \mathbb Z}w\left( \frac{d[\ell,s^2]}{X}\right)\sum_{p\nmid 2sN_E } \left( \frac{d[\ell,s^2]}{p}\right)\frac{\lambda_E(p) \log p}{p^{1/2}} \widehat \phi\left( \frac{ \log p}{L} \right)  \\
& \ll  \frac 1{W^*(X)} \sum_{\substack{ \ell \mid N_E \\ s >S}} \sum_{0\neq d \in \mathbb Z} w\left( \frac{d[\ell,s^2]}{X}\right) \big(\log (2|d|\ell s N_EX)\big)^3\ll_{\varepsilon} (XN_E)^{\varepsilon}  \sum_{s>S} \frac {(\log s)^3}{s^2} \\& \ll_{\varepsilon} \frac{N_E^{\varepsilon}X^{\varepsilon}(\log S)^3}S.
\end{align*}
The rest of the proof is similar to that of Lemma \ref{lemma:initialPoisson}, the main ingredient being Poisson Summation.
\end{proof}

We now handle the terms with $s\leq S$ in $S_{\text{odd}}^*$.
\begin{lemma}
\label{lemma:small s}
Assume ECRH, fix $\varepsilon>0$ and suppose that $\sigma:= \text{sup}(\text{supp} \widehat \phi)< \infty $. Then, for any $1 \leq S \leq X^{2}$, we have that\footnote{This range can be replaced by $1 \leq S \leq X^{M}$, for any fixed $M\in \mathbb N$. However, the important range for our analysis is $1 \leq S \leq X^2$.}
$$ \sum_{\substack{\ell \mid N_E \\ s \leq S}} \frac{\mu(\ell)\mu(s)}{[\ell, s^2]} \sum_{t \in \mathbb Z} \sum_{\substack{p \nmid 2 sN_E }} \frac{ \left( \frac{-[\ell,s^2]t }p\right)\overline{\epsilon_p} \lambda_E(p) \log p}{p}\widehat \phi\left( \frac{\log p}{L} \right)  \widehat w\left( \frac{Xt}{p[\ell, s^2]}\right) \ll_{\varepsilon,E} SX^{\sigma-1+\varepsilon}.$$
\end{lemma}

\begin{proof}
We first add back the primes dividing $2sN_E$, at the cost of an error term which is 
$$ \ll \sum_{\substack{  \ell \mid N_E \\ s \leq S}} \frac 1{[\ell,s^2]}\sum_{\substack{p \mid 2 sN_E }} \frac{\log p}p \frac{p[\ell,s^2]}X \ll_{\varepsilon} N_E^{\varepsilon}SX^{-1} \log (2N_ES).  $$
We then follow the steps of Lemma \ref{lemma:ECRHsmallsupport}. The sum we are interested in equals
$$\sum_{\substack{\ell \mid N_E \\ s \leq S}} \frac{\mu(\ell)\mu(s)}{[\ell, s^2]} \sum_{0\neq t \in \mathbb Z} P'_{t,\ell,s}+O_{\varepsilon}\big(N_E^{\varepsilon}SX^{-1} \log (2N_ES)\big),$$
where (write $c_E:=N_E/(2\pi e)^2$)
$$ P'_{t,\ell,s}:= \int_{1}^{(c_EX^2)^{\sigma}} \frac{ \widehat \phi\left( \frac{\log x}L\right) \widehat w\left( \frac{Xt}{{x}[\ell,s^2]}\right)}{x} dS_{t,\ell,s}(x), $$
with
$$ S_{t,\ell,s}(y)= \sum_{\substack{p\leq y}} \left( \frac {-[\ell,s^2]t} p\right) \overline{\epsilon_p} \lambda_E(p) \log p \ll y^{\frac 12} (\log (2y))\log(2N_E|t|s\ell y). $$
Performing integration by parts, we obtain the bound
\begin{equation}
 P'_{t,\ell,s} \ll \int_{1}^{(c_EX^2)^{\sigma}}   \left[   \frac{ \Big| \widehat w\left( \frac{Xt}{x[\ell,s^2]} \right)\Big| }{x^{2}} + \frac{ \Big|\widehat w'\left( \frac{Xt}{{x}[\ell,s^2]} \right)\Big| X|t|}{x^{3}[\ell,s^2]} \right]x^{\frac 12} \big(\log (2N_Ex|t|s\ell)\big)^2dx.
\end{equation}
Recall that
$$ \sum_{0\neq t \in \mathbb Z}\left| \widehat w\left( \frac{Xt}{x[\ell,s^2]} \right)\right| \big(\log (2|t|)\big)^2 \ll \frac{x [\ell,s^2]}{X}\big(\log(2xX\ell s)\big)^2;$$ $$ \sum_{0\neq t \in \mathbb Z} X|t|\left|\widehat w'\left( \frac{Xt}{x[\ell,s^2]} \right)\right|\big(\log (2|t|)\big)^2 \ll \frac{(x[\ell,s^2])^2}X\big(\log (2xX\ell s)\big)^2, $$
from which we obtain the bound
\begin{align*}
\sum_{\substack{\ell \mid N_E \\ s \leq S}} \frac{\mu(\ell)\mu(s)}{[\ell, s^2]} \sum_{0\neq t \in \mathbb Z} P'_{t,\ell,s}&\ll  \sum_{\substack{\ell \mid N_E \\ s \leq S}} \frac{1}{[\ell, s^2]} \int_{1}^{(c_EX^2)^{\sigma}}   \left[   \frac{ [\ell,s^2] }{x X} + \frac{ [\ell,s^2] }{x X} \right]x^{\frac 12} \big(\log (2xX\ell s)\log (2N_Exs\ell)\big)^2dx \\&\ll_{\varepsilon,E} S(\log (2S))^4X^{\sigma-1+\varepsilon} .
\end{align*}  
This concludes the proof.
\end{proof}

We summarize the current section in the following theorem.

\begin{theorem}\label{SqfreeSodd}
Assume RH and ECRH, and suppose that $\sigma:= \text{sup}(\text{supp} \widehat \phi)< 1 $. Then, for any fixed $\varepsilon>0$, we have the bound
$$  S^*_{\text{odd}} \ll_{\varepsilon,E}  X^{\frac{\sigma-1}2+\varepsilon}.$$
\end{theorem}

\begin{proof}
Take $S=X^{\frac{1-\sigma}2}$ in Lemmas \ref{lemma:secondPoisson} and \ref{lemma:small s}.
\end{proof}

Finally, we complete the proof of Theorem \ref{third main theorem}.

\begin{proof}[Proof of Theorem \ref{third main theorem}]
The desired result follows directly from Corollary \ref{corollary 1 level densities}, Lemma \ref{lemma count of squarefree} and Theorem \ref{SqfreeSodd} (cf.\ the proof of Theorems \ref{main theorem} and \ref{second main theorem}).
\end{proof}

\appendix
\section{The ratios conjecture's prediction}
\label{Appendix A}

The lower order terms in the $1$-level density for the family of quadratic twists of a given elliptic curve $E$ with prime conductor and even sign of the functional equation was computed by Huynh, Keating and Snaith in \cite{HKS} using the Ratios Conjecture techniques of \cite{CFZ} and \cite{CS}. In this appendix we perform the corresponding calculations in the context of our weighted family of all quadratic twists coprime to the (not necessarily prime) conductor $N_E$ of the given elliptic curve $E$. Throughout this section we assume the Riemann Hypothesis for all $L$-functions that we encounter. As in Sections \ref{section all d} and \ref{section square free}, every error term in this section is allowed to depend on $\phi$ and $w$, but we now allow an additional dependence on $E$.

\begin{theorem}
\label{oneleveldensityresult}
Fix $\varepsilon>0$. Let $E$ be an elliptic curve defined over $\Q$ with conductor $N_E$. Let $w$ be a nonnegative Schwartz function on $\R$ which is not identically zero and let $\phi$ be an even Schwartz function on $\R$ whose Fourier transform has compact support. Assuming GRH and Conjecture \ref{ratiosconjecture} (the Ratios Conjecture for our family), the $1$-level density for the zeros of the family of $L$-functions attached to the quadratic twists of $E$ coprime to $N_E$ is given by
\begin{align*}
\mathcal D(\phi;X)=& \frac {\widehat \phi (0)}{LW(X)} \sumt\log\bigg(\frac{N_{E}d^2}{(2\pi)^2}
\bigg) + \frac{1}{2\pi}\int_{\mathbb R}
\phi\left(\frac{tL}{2\pi}\right)\bigg[\frac{\Gamma'(1+it) }{\Gamma(1+it) }\\
 & +\frac{\Gamma'(1-it)}{\Gamma(1-it)} + 2 \bigg(-\frac{\zeta'(1+2it)}{\zeta(1+2it)} +\frac{L'\left(1+2it,{\rm Sym}^2E\right)}{L\left(1+2it,{\rm Sym}^2E\right)}+A_{\alpha,E}(it,it)\bigg)- \frac{1}{it} \bigg]dt\\
&+  \frac{\phi(0)}{2} + O_{\varepsilon}\big(X^{-1/2+\varepsilon}\big),
\end{align*}
where $^*$ indicates that we are summing over square-free $d$, the functions $L$, $W$ and $\widetilde w$ are defined by \eqref{Ldefinition}, \eqref{W-sum} and \eqref{wtildeintro} respectively, $L\left(s,{\rm Sym}^2E\right)$ is the symmetric square $L$-function of $E$ (see \eqref{defofsymsqalpha}), and the function $A_{\alpha,E}$ is defined by \eqref{AALPHAE} (see also \eqref{vnmid}, \eqref{vmid}, \eqref{defofy} and \eqref{defnofae}). 
\end{theorem}

Rewriting the rather complicated expression for the function $A_{\alpha,E}$, we obtain the following alternative formula for the sum of the second and third terms appearing in Theorem \ref{oneleveldensityresult}. 

\begin{theorem}
\label{extendedoneleveldensityresult}
Fix $\varepsilon>0$. Let $E$ be an elliptic curve defined over $\Q$ with conductor $N_E$, and let $\phi$ be an even Schwartz function on $\R$ whose Fourier transform has 
compact support. We have the following expression for the sum of the second and third terms appearing in Theorem \ref{oneleveldensityresult}: 
\begin{align*}
\frac{1}{2\pi}&\int_{\mathbb R}
\phi\left(\frac{tL}{2\pi}\right)\bigg[\frac{\Gamma'(1+it) }{\Gamma(1+it) }
 +\frac{\Gamma'(1-it)}{\Gamma(1-it)} + 2 \bigg(-\frac{\zeta'(1+2it)}{\zeta(1+2it)} +\frac{L'\left(1+2it,{\rm Sym}^2E\right)}{L\left(1+2it,{\rm Sym}^2E\right)}\\&+A_{\alpha,E}(it,it)\bigg)- \frac{1}{it} \bigg]dt+  \frac{\phi(0)}{2}=-\frac{2} L\int_0^{\infty} \left( \frac{\widehat \phi(x/L) e^{-x} }{1-e^{-x}} - \widehat \phi(0) \frac{e^{-x}}x \right)dx\nonumber\\
&-\frac {2}{L}\sum_{\ell=1}^\infty \sum_{p} \frac{\left(\alpha_E(p)^{2\ell}+\beta_E(p)^{2\ell}\right) \log p}{p^{\ell}} \widehat \phi\left( \frac{2\ell\log p}{L} \right) \left( 1+\frac{\psi_{N_E}(p)}p \right)^{-1}+O_{\epsilon}\big(X^{-1+\varepsilon}\big), 
\end{align*}
where $\psi_{N_E}$ is the principal Dirichlet character modulo $N_E$ and the function $L$ is defined by \eqref{Ldefinition}.
\end{theorem}

\begin{remark}
The proof of Theorem \ref{oneleveldensityresult} can, with only minimal changes, be turned into a proof of the corresponding result for the weighted family of all \emph{square-free} quadratic twists coprime to the conductor $N_E$ of the given elliptic curve $E$. We record this result (combined with Theorem \ref{extendedoneleveldensityresult}) here for convenience:

Let $E$, $w$, $\phi$ and $\varepsilon$ be as in Theorem \ref{oneleveldensityresult}. Assuming GRH and Conjecture \ref{ratiosconjecture}, the $1$-level density for the zeros of the family of $L$-functions attached to the square-free quadratic twists of $E$ coprime to $N_E$ is given by
\begin{align*}
\mathcal D^*(\phi;X)&=\frac {\widehat \phi(0)}{LW^*(X)}  \sumtstar \log \bigg(\frac{N_{E} d^2}{(2\pi)^2}\bigg)-\frac{2} L\int_0^{\infty} \left( \frac{\widehat \phi(x/L) e^{-x} }{1-e^{-x}} - \widehat \phi(0) \frac{e^{-x}}x \right)dx\nonumber\\
&-\frac {2}{L}\sum_{\ell=1}^\infty \sum_{p} \frac{\left(\alpha_E(p)^{2\ell}+\beta_E(p)^{2\ell}\right) \log p}{p^{\ell}} \widehat \phi\left( \frac{2\ell\log p}{L} \right) \left( 1+\frac {\psi_{N_E}(p)}p \right)^{-1} 
+O_{\varepsilon}\big(X^{-1/2+\varepsilon}\big),
\end{align*}
where the functions $L$ and $W^*$ are defined by \eqref{Ldefinition} and \eqref{Wstar-sum} respectively.  
\end{remark}

\begin{remark}
We prove Theorems \ref{oneleveldensityresult} and \ref{extendedoneleveldensityresult} for Schwartz test functions $\phi$ for which the Fourier transforms have compact support. This is a more restricted class of test functions than is typically used in results based on the Ratios Conjecture. However, this class is more than sufficient for our purposes in the present paper. Let us also point out that even though we could, with more work, prove Theorem \ref{oneleveldensityresult} for a larger class of test functions, we are at present not aware of any proof of Theorem \ref{extendedoneleveldensityresult} which avoids the assumption that the test functions $\phi$ have compactly supported Fourier transforms.   
\end{remark}

\subsection{Proof of Theorem \ref{oneleveldensityresult}}

To begin, we derive the appropriate version of the Ratios Conjecture. Thus we consider the sum
\begin{equation} \label{ralphgam}
R(\alpha,\gamma):=\frac1{W(X)}\sumt \frac{L\left(\frac{1}{2}+\alpha,E_d\right)}{L\left(\frac{1}{2}+\gamma,E_d\right)}.
\end{equation}
In order to rewrite the expression for $R(\alpha,\gamma)$ we recall two well-known formulas. The first formula is
\begin{equation} \label{lemobius}
\frac{1}{L(s,E_d)} = \sum_{n=1}^\infty
\frac{\mu_{E}(n)\chi_d(n)}{n^s},
\end{equation}
where $\mu_{E}$ is the multiplicative function given by
\begin{equation*} 
\mu_{E}(p^k) =
\begin{cases}
-\lambda_E(p) & \mbox{~if~} k = 1,\\
\psi_{N_E}(p) & \mbox{~if~} k = 2,\\
0 & \mbox{~if~} k > 2,
\end{cases}
\end{equation*}
and $\psi_{N_E}$ is the principal Dirichlet character modulo $N_E$. The second formula is the approximate functional equation for $L(s,E_d)$:
\begin{equation}
L\left(s, E_d\right) = \sum_{n<x} \frac{\lambda_E(n)\chi_d(n)}{n^s}+\epsilon_{E_d}\left(\frac{\sqrt{N_{E}}|d|}{2\pi}\right)^{1-2s}\frac{\Gamma\left(\frac{3}{2}-s\right)}{\Gamma\left(\frac{1}{2}+s\right)} \sum_{n<y}
\frac{\lambda_E(n)\chi_d(n)}{n^{1-s}}+\:{\rm Error},\label{approxeq}
\end{equation}
where $xy=d^2/(2\pi)$. As a part of the Ratios Conjecture recipe, we will in the following calculations  disregard the error term and complete the sums (i.e.\ replace $x$ and $y$ with infinity).

Following \cite{CFZ}, we  replace the
numerator of $\eqref{ralphgam}$ with the approximate functional equation $\eqref{approxeq}$ (modified as above) and the
denominator of $\eqref{ralphgam}$ with $\eqref{lemobius}$. We will focus on
the principal sum from the approximate functional equation in $\eqref{approxeq}$  evaluated at $s=\frac{1}{2}+\alpha$, which gives the contribution
\begin{equation} \label{ragone}
R_1(\alpha, \gamma):=\frac1{W(X)}\sumt \sum_{h,m} \frac{\lambda_E(m)\mu_{E}(h)\chi_d(hm)}{m^{\frac{1}{2} +
\alpha}h^{\frac{1}{2}+\gamma}}
\end{equation}
to \eqref{ralphgam}. We also have to consider the sum coming from
replacing the dual sum from the approximate functional equation (the second sum in $\eqref{approxeq}$)
in $\eqref{ralphgam}$, namely the sum
\begin{equation} \label{ragtwo}
R_2(\alpha, \gamma):=\frac1{W(X)}\sumt \epsilon_{E_d}X_{E_d}\left(\tfrac{1}{2} + \alpha\right)\sum_{h,m} \frac{\lambda_E(m)\mu_{E}(h)\chi_d(hm)}{m^{\frac{1}{2}-\alpha}h^{\frac{1}{2}+\gamma}},
\end{equation}
where
\begin{equation} \label{xe}
X_{E_d}(s):=\frac{\Gamma\left(\frac{3}{2}-s\right)}{\Gamma\left(\frac{1}{2}+s\right)}\left(\frac{\sqrt{N_{E}}|d|}{2\pi}\right)^{1-2s}.
\end{equation}
However, the next step is to replace the root numbers in \eqref{ragtwo} with their expected value when averaged over the
family. In this family the expected value of the root numbers is zero; thus we replace $R_2(\alpha,\gamma)$ by zero. 

Continuing the Ratios Conjecture procedure, we replace $\wt\chi_d(hm)$ in $\eqref{ragone}$ with its average over the set being summed.
From Lemma \ref{lemmaweighted} and Remark \ref{whatiswx}, we have that
\begin{align*}
&\frac{1}{W(X)}\sumt\chi_d(hm) \\ 
&=\begin{cases}\prod_{p\mid hm} \left( \frac{p}{p+1}\right)\prod_{\substack{p\mid (hm,N_E)}} \left( 1+\frac{1}{p} \right) +O_{h,m,\varepsilon}(X^{-1+\varepsilon}) & {\rm if} \: hm = \square,\\
O_{h,m,\varepsilon}(X^{-1+\varepsilon})  & {\rm otherwise}.
\end{cases}
\end{align*} 
So the main contribution to the sum in $\eqref{ragone}$ occurs when $hm=\square$ and we will disregard the non-square terms and the error terms. Hence, following the recipe, we have that $\eqref{ragone}$ is replaced with
\begin{equation*}
\widetilde R_1(\alpha, \gamma)=
\sum_{hm=\square} \frac{\lambda_E(m)\mu_{E}(h)}{m^{\frac{1}{2} +
\alpha}h^{\frac{1}{2}+\gamma}}\prod_{p\mid hm} \left( \frac{p}{p+1}\right) \prod_{\substack{p\mid (hm,N_E)}} \left( 1+\frac{1}{p} \right),
\end{equation*}
and writing this as an Euler product gives
\begin{align*}
\widetilde R_1(\alpha, \gamma)=
\prod_{p\nmid N_E}\Bigg(1+\frac{p}{p+1}\sum_{\substack{e,k \geq 0 \\ e+k>0 \\ e+k \:{\rm even}}}\frac{\lambda_E(p^e)\mu_E(p^k)}{p^{e(\frac{1}{2}+\alpha)+k\left(\frac{1}{2}+\gamma\right)}}\Bigg) \prod_{p\mid N_E}\Bigg(\sum_{\substack{e,k \geq 0 \\ e+k \:{\rm even}}}\frac{\lambda_E(p^e)\mu_E(p^k)}{p^{e(\frac{1}{2}+\alpha)+k\left(\frac{1}{2}+\gamma\right)}}\Bigg).
\end{align*}
Since $\mu_E(p^k)=0$ for $k>2$, we need only consider the cases $k=0,1$ or 2 and we define
\begin{align}
V_{\nmid}(\alpha,\gamma):=& \prod_{p\nmid N_E} \left(1 + \frac{p}{p+1}\left(\sum_{e=1}^{\infty}\frac{\lambda_E(p^{2e})}{p^{e(1+2\alpha)}}-\frac{\lambda_E(p)}{p^{1+\alpha+\gamma}}\sum_{e=0}^\infty\frac{\lambda_E(p^{2e+1})}{p^{e(1+2\alpha)}} +\frac{1}{p^{1+2\gamma}}\sum_{e=0}^\infty\frac{\lambda_E(p^{2e})}{p^{e(1+2\alpha)}}\right)\right)\label{vnmid}
\end{align}
and
\begin{align}
V_{\mid}(\alpha,\gamma):=& \prod_{p\mid N_E}\left(\sum_{e=0}^{\infty}\frac{\lambda_E(p^{2e})}{p^{e(1+2\alpha)}}-\frac{\lambda_E(p)}{p^{1+\alpha+\gamma}}\sum_{e=0}^\infty\frac{\lambda_E(p^{2e+1})}{p^{e(1+2\alpha)}}\right).\label{vmid}
\end{align}

In the Euler products in $\eqref{vnmid}$ and $\eqref{vmid}$, we factor out the terms that contribute poles and zeros to $\widetilde R_1(\alpha,\gamma)$ as $\alpha, \gamma \rightarrow 0$. We also factor out the symmetric square $L$-function associated with $L(s,E)$. Restricting our attention to $\alpha$ and $\gamma$ satisfying $\frac{-1+\delta}{4}< \Re(\alpha)< \frac{1}{4}$ and $\frac{1}{\log X} \ll \Re(\gamma)<\frac{1}{4}$ for some small $\delta>0$, we have
\begin{equation} \label{vnmidpoles}
V_{\nmid}(\alpha,\gamma)=\prod_{p\nmid N_E}\left(1+\frac{\lambda_E(p^2)}{p^{1+2\alpha}}-\frac{\lambda_E(p^2)+1}{p^{1+\alpha+\gamma}}+\frac{1}{p^{1+2\gamma}}+
O\left(\frac{1}{p^{1+\delta}}\right)\right).
\end{equation}
Furthermore, recalling that the symmetric square $L$-function can be written in the form
\begin{align}
L\left(s,{\rm Sym}^2 E\right)&=
\prod_{p}\left(1-\frac{\alpha_E(p)^2}{p^s}\right)^{-1}\left(1-\frac{\alpha_E(p)\beta_E(p)}{p^s}\right)^{-1}\left(1-\frac{\beta_E(p)^2}{p^s}\right)^{-1} \label{defofsymsqalpha}
\\&= \prod_{p\mid N_E}\left(1-\frac{\lambda_E(p^2)}{p^s}\right)^{-1}\prod_{p\nmid N_E}\left(1-\frac{\lambda_E(p^2)}{p^s}+\frac{\lambda_E(p^2)}{p^{2s}}-\frac{1}{p^{3s}}\right)^{-1},\label{defofsymsq}
\end{align}
we find that 
$$L\left(1+2\alpha,{\rm Sym}^2 E\right)=\prod_p\left(1+\frac{\lambda_E(p^2)}{p^{1+2\alpha}}+O\left(\frac{1}{p^{1+\delta}}\right)\right)$$
and
$$\frac{1}{\zeta(1+\alpha+\gamma)L\left(1+\alpha+\gamma,{\rm Sym}^2 E\right)}=\prod_p\left(1-\frac{\lambda_E(p^2)+1}{p^{1+\alpha+\gamma}}+O\left(\frac{1}{p^{3/2}}\right)\right).$$
 Finally, since there are only finitely many primes dividing $N_E$, it is clear that the factor $\zeta(1+2\gamma)$ will account for the divergence of the term
$\frac{1}{p^{1+2\gamma}}$ in $\eqref{vnmidpoles}$. Hence we can write
 \begin{align*}
 \widetilde R_1(\alpha,\gamma)= 
 V_{\nmid}(\alpha,\gamma)V_{\mid}(\alpha,\gamma) =
 Y_E(\alpha,\gamma)A_E(\alpha,\gamma),
 \end{align*}
where
\begin{equation}
Y_E(\alpha,\gamma):=\frac{\zeta(1+2\gamma)L\left(1+2\alpha,{\rm Sym}^2 E\right)}{\zeta(1+\alpha+\gamma)L\left(1+\alpha+\gamma,{\rm Sym}^2 E\right)},\label{defofy}
\end{equation}
and
\begin{equation}
A_E(\alpha,\gamma):=Y_E(\alpha,\gamma)^{-1}V_{\nmid}(\alpha,\gamma)V_{\mid}(\alpha,\gamma)\label{defnofae}
\end{equation}
is analytic as $\alpha,\gamma\rightarrow 0$. Thus the Ratios Conjecture for our weighted family of elliptic curve $L$-functions is given by:

\begin{conjecture} \label{ratiosconjecture}
Let $\varepsilon>0$ and let $w$ be a nonnegative Schwartz function on $\R$ which is not identically zero. Let $\delta>0$ and suppose that the complex numbers $\alpha$ and $\gamma$ satisfy $\frac{-1+\delta}{4}< \Re(\alpha)< \frac{1}{4}$, $\frac{1}{\log X} \ll\Re(\gamma)<\frac{1}{4}$ and $\Im(\alpha),\Im(\gamma) \ll X^{1-\varepsilon}$. Then we have that
\begin{equation*}
 \xf\sumt \frac{L(\frac{1}{2} +\alpha,E_d)}{L(\frac{1}{2} + \gamma,E_d)} = 
 Y_E(\alpha,\gamma)A_E(\alpha,\gamma)+ O_\varepsilon\big(X^{-1/2+\varepsilon}\big),
\end{equation*}
where $Y_E(\alpha,\gamma)$ is defined in $\eqref{defofy}$ and $A_E(\alpha,\gamma)$ is defined in $\eqref{defnofae}$.\footnote{We stress that the error term $O_{\varepsilon}(X^{-1/2+\varepsilon})$ is part of the statement of the Ratios Conjecture. Let us also point out that the condition on the imaginary parts of $\alpha$ and $\gamma$ is not used in the derivation of Conjecture \ref{ratiosconjecture}, but is included as a plausible (and by now standard) condition under which conjectures produced by the Ratios Conjecture recipe are expected to hold.}
\end{conjecture}

We require the family average of the logarithmic derivative of the $L$-functions $L(s,E_d)$ in our calculation of the $1$-level density. Thus we differentiate the result of Conjecture \ref{ratiosconjecture} with respect to $\alpha$. First we define
\begin{equation}\label{AALPHAE}
A_{\alpha,E}(r,r) := \frac{\partial}{\partial
\alpha}A_E(\alpha,\gamma)\bigg|_{\alpha=\gamma=r}.
\end{equation}

\begin{lemma} \label{ratiostheorem} 
Let $\varepsilon>0$ and let $w$ be a nonnegative Schwartz function on $\R$ which is not identically zero. Suppose that $r \in \CC$ satisfies $\frac{1}{\log X} \ll\Re(r)<\frac{1}{4}$ and $\Im(r)\ll X^{1-\varepsilon}$. Then, assuming ECRH and Conjecture \ref{ratiosconjecture}, we have that
\begin{multline}
 \xf\sumt \frac{L'(\frac{1}{2} + r,E_d)}{L(\frac{1}{2} + r,E_d)} \\
= -\frac{\zeta'(1+2r)}{\zeta(1+2r)} +\frac{L'\left(1+2r,{\rm Sym}^2 E\right)}{L\left(1+2r,{\rm Sym}^2 E\right)}+A_{\alpha,E}(r,r)+ O_{\varepsilon}\big(X^{-1/2+\varepsilon}\big).\label{ratiostheoremthree}
\end{multline}
\end{lemma}

\begin{proof}
We have that
\begin{equation*}
\frac{\partial}{\partial \alpha}Y_E(\alpha, \gamma)A_E(\alpha,
\gamma)\bigg|_{\alpha=\gamma=r} = -\frac{\zeta'(1+2r)}{\zeta(1+2r)} + \frac{L'\left(1+2r,{\rm Sym}^2 E\right)}{L\left(1+2r,{\rm Sym}^2 E\right)}+A_{\alpha,E}(r,r),
\end{equation*}
which gives the main term in \eqref{ratiostheoremthree}. The fact that the error term remains the same under differentiation follows immediately from a standard argument based on Cauchy's integral formula for derivatives. 
\end{proof}

\begin{proof}[Proof of Theorem \ref{oneleveldensityresult}]
We recall from $\eqref{alternative1levelformula}$ that
\begin{equation*}
\mathcal D(\phi;X)=\xf\sumt \sum_{\gamma_d}\phi\left(\gamma_d\frac{L}{2\pi}\right).
\end{equation*}
Hence, by the argument principle, we have that
\begin{equation} \label{cauchydens}
\mathcal D(\phi;X)=\xf\sumt \frac{1}{2\pi
i}\left(\int_{(c)} - \int_{(1-c)}\right)\frac{L'(s,E_d)}{L(s,E_d)}
\phi\left(\frac{-iL}{2\pi}\left(s-\frac{1}{2}\right)\right) ds
\end{equation}
with $\frac{1}{2}+\frac{1}{\log X}<c<\frac34$.
 
For the integral on the line with real part $1-c$, making the change of variable $s \rightarrow 1-s$ and recalling that $\phi$
is even, we find that it equals
\begin{align}
\xf\sumt \frac{1}{2\pi i} \int_{(c)} \frac{L'(1-s,E_d)}{L(1-s,E_d)}
\phi\left(\frac{-iL}{2\pi}\left(s-\frac{1}{2}\right)\right)ds.
\label{loneminuss}
\end{align}
Also, using the functional equation $\eqref{functionaltwist}$ and $\eqref{xe}$, we obtain
\begin{equation}
 \label{labellprimeoverl}
\frac{L'(s,E_d)}{L(s,E_d)} = \frac{X_{E_d}'(s)}{X_{E_d}(s)} -
\frac{L'(1-s,E_d)}{L(1-s,E_d)}\,.
\end{equation}
Hence, from $\eqref{loneminuss}$ and $\eqref{labellprimeoverl}$ and making the change of variable $s = 1/2 + r$, we have that $\eqref{cauchydens}$ becomes
\begin{align}
\mathcal D(\phi;X)
=& \xf\sumt \frac{1}{2\pi i}\int_{(c-1/2)} \bigg[2 \frac{L'(1/2 +
r,E_d)}{L(1/2 + r,E_d)} - \frac{X_{E_d}'(1/2+r)}{X_{E_d}(1/2+r)}\bigg] \phi\left(\frac{iLr}{2\pi}\right)dr.
\label{checknine}
\end{align}
We bring the summation
inside the integral and substitute
\begin{equation*}
\xf\sumt\frac{L'(1/2 +r,E_d)}{L(1/2 + r,E_d)}
\end{equation*}
with the right-hand side of $\eqref{ratiostheoremthree}$. Note that this substitution is a priori valid only for $r$ with $\Im(r)<X^{1-\varepsilon}$. However, since $\widehat \phi$ is assumed to have compact support  on $\R$, it is clear that $\phi\left(\frac{iLr}{2\pi}\right)$ is rapidly decaying as $|\Im(r)|\to\infty$. This fact, together with standard estimates of the logarithmic derivative of $L$-functions in the half-plane $\Re(s)>\frac12$ (see,\ e.g.,\ \cite[Thm.\ 5.17]{IK}), make it possible to bound the tail of the integral in \eqref{checknine} (where we cannot apply Lemma \ref{ratiostheorem}) by $O_{\varepsilon}\left(X^{-1+\varepsilon}\right)$. Furthermore, applying the same tools to bound also the tail of the integral in \eqref{RATIOSONELEV}, we arrive at 
\begin{align}
\mathcal D(\phi;X)
=&\xf\sumt \frac{1}{2\pi i}\int_{(c-1/2)}  \bigg[-2 \frac{\zeta'(1+2r)}{\zeta(1+2r)} +2\frac{L'\left(1+2r,{\rm Sym}^2 E\right)}{L\left(1+2r,{\rm Sym}^2 E\right)}\nonumber\\
+& 2A_{\alpha,E}(r,r)-\frac{X_{E_d}'(1/2+r)}{X_{E_d}(1/2+r)}\bigg] \phi\left(\frac{iLr}{2\pi}\right)dr + O_{\varepsilon}\big(X^{-1/2+\varepsilon}\big).\label{RATIOSONELEV}
\end{align}

We now move the contour of integration from $\Re(r)=c-1/2=c'$ to
$\Re(r)=0$. However, the function
\begin{equation*} 
2 \left(-\frac{\zeta'(1+2r)}{\zeta(1+2r)} +\frac{L'\left(1+2r,{\rm Sym}^2E\right)}{L\left(1+2r,{\rm Sym}^2E\right)}+A_{\alpha,E}(r,r)\right)-
\frac{X_{E_d}'(1/2+r)}{X_{E_d}(1/2+r)}
\end{equation*}
has a pole at $r=0$ with residue 1. Thus, by Cauchy's Theorem, we have that
\begin{align*} \nonumber
\mathcal D(\phi;X)=&\xf\sumt \frac{1}{2\pi}
\int_{\mathbb R} \bigg[- 2 \frac{\zeta'(1+2it)}{\zeta(1+2it)}+2\frac{L ' \left(1+2it,{\rm Sym}^2 E\right)}{L\left(1+2it,{\rm Sym}^2 E\right)}\\
 & + 2A_{\alpha,E}(it,it)  +\log\left(\frac{N_{E}d^2}{(2\pi)^2}\right)+\frac{\Gamma'(1-it)}{\Gamma(1-it)}
+ \frac{\Gamma'(1+it)}{\Gamma(1+it)} - \frac{1}{it}\bigg] \,
\phi\left(\frac{tL}{2\pi}\right) \, dt \\
& +\xf \sumt \frac{1}{2 \pi i} \int_{(c')}   \frac{\phi\left(\frac{iLr}{2\pi}\right)}{r}  \,dr + O_{\varepsilon}\big(X^{-1/2+\varepsilon}\big).
\end{align*}
Finally, we note that 
\begin{equation*}
 \phi(0)= \frac{1}{2 \pi i} \int_{(c')}   \frac{\phi\left(\frac{iLr}{2\pi}\right)}{r} dr - \frac{1}{2 \pi i} \int_{(-c')}  \frac{\phi\left(\frac{iLr}{2\pi}\right)}{r} dr =
\frac{2}{2 \pi i} \int_{(c')}  \frac{\phi\left(\frac{iLr}{2\pi}\right)}{r}dr ,
\end{equation*}
which completes the proof.
\end{proof}

\begin{remark}\label{RATIOSONELEV2}
Note that, using \cite[Lemme I.2.1]{Me}, we also get the following useful formula as a corollary of  equation \eqref{RATIOSONELEV},
\begin{align*}
\mathcal D&(\phi;X)=\frac {\widehat \phi(0)}{LW(X)}  \sumt \log \bigg(\frac{N_{E}d^2}{(2\pi)^2}\bigg)-\frac{2} L\int_0^{\infty} \left( \frac{\widehat \phi(x/L) e^{-x} }{1-e^{-x}} - \widehat \phi(0) \frac{e^{-x}}x \right)dx\\
&+\frac{1}{\pi i}\int_{(c')}  \bigg[-\frac{\zeta'(1+2r)}{\zeta(1+2r)}+\frac{L'\left(1+2r,{\rm Sym}^2 E\right)}{L\left(1+2r,{\rm Sym}^2 E\right)}
+A_{\alpha,E}(r,r)\bigg] \phi\left(\frac{iLr}{2\pi}\right)dr+ O_{\varepsilon}\big(X^{-1/2+\varepsilon}\big).
\end{align*} 
\end{remark}

\subsection{Proof of Theorem \ref{extendedoneleveldensityresult}}

We determine the contribution of $A_{\alpha,E}$ to Theorem \ref{oneleveldensityresult} and Remark \ref{RATIOSONELEV2} by first obtaining a useful expansion for it. Note that $A_E(r,r)=1$ and hence, following \cite{HMM}, from $\eqref{defofy}$ and $\eqref{defnofae}$ we have that 
\begin{align}
A_{\alpha,E}(r,r)
=&\sum_{p\mid N_E}\log p \left[\frac{\frac{\lambda_E(p)^2}{p^{1+2r}}}{1-\frac{\lambda_E(p)^2}{p^{1+2r}}}
-\frac{\frac{1}{p^{1+2r}}}{1-
\frac{1}{p^{1+2r}}}-\sum_{e=1}^\infty \frac{\lambda_E(p^{2e})}{p^{e(1+2r)}}\right]\nonumber \\
+&\sum_{p\nmid N_E}\log p\Bigg[\frac{\frac{\lambda_E(p^2)}{p^{1+2r}}-\frac{2\lambda_E(p^2)}{p^{2(1+2r)}}+
\frac{3}{p^{3(1+2r)}}}{1-\frac{\lambda_E(p^2)}{p^{1+2r}}+\frac{\lambda_E(p^2)}{p^{2(1+2r)}}-\frac{1}{p^{3(1+2r)}}}
-\frac{\frac{1}{p^{1+2r}}}{1-\frac{1}{p^{1+2r}}} \nonumber\\ -&\sum_{e=0}^\infty \frac{\lambda_E(p^{2e+2}) -
\lambda_E(p^{2e})}{p^{(e+1)(1+2r)}} +\frac{1}{p
+1}\sum_{e=0}^\infty \frac{\lambda_E(p^{2e+2})-\lambda_E(p^{2e})}{p^{(e+1)(1+2r)}}\Bigg]. \label{Aalphabig}
\end{align}
We can express the logarithmic derivative of $\zeta(s)$ as  
\begin{equation}
\frac{\zeta'(1+2r)}{\zeta(1+2r)}=-\sum_{p} \log p\left[\frac{\frac{1}{p^{1+2r}}}{1-\frac{1}{p^{1+2r}}}\right]. \label{zetaoverzeta}
\end{equation} 
As for that of $L(s,{\rm Sym}^2 E)$, by \eqref{defofsymsqalpha} and \eqref{defofsymsq}, we obtain
\begin{align}
\frac{L'\left(1+2r,{\rm Sym}^2 E\right)}{L\left(1+2r,{\rm Sym}^2 E\right)}
=&-\sum_{p\mid N_E} \log p 
\left[\frac{\frac{\lambda_E(p)^2}{p^{1+2r}}}{1-\frac{\lambda_E(p)^2}{p^{1+2r}}} \right]-\sum_{p\nmid N_E} 
\log p \left[\frac{\frac{\lambda_E(p^2)}{p^{1+2r}}-\frac{2\lambda_E(p^2)}{p^{2(1+2r)}}+\frac{3}{p^{3(1+2r)}}}{1-
\frac{\lambda_E(p^2)}{p^{1+2r}}+\frac{\lambda_E(p^2)}{p^{2(1+2r)}}-\frac{1}{p^{3(1+2r)}}} \right]\nonumber\\
=&-\sum_{p\mid N_E} \log p \sum_{\ell=1}^\infty\frac{\alpha_E(p)^{2\ell}}{p^{\ell(1+2r)}}
-\sum_{p\nmid N_E}\sum_{\ell=1}^\infty \frac{\left(\alpha_E(p)^{2\ell}+\beta_E(p)^{2\ell}+1\right)\log p}{p^{\ell(1+2r)}}.
\label{lsymsderiv}
\end{align}
Thus, we have that $\eqref{Aalphabig}$ equals 
\begin{align}
A_{\alpha,E}(r,r)&=\frac{\zeta'(1+2r)}{\zeta(1+2r)}-\frac{L'\left(1+2r,{\rm Sym}^2 E\right)}{L\left(1+2r,{\rm Sym}^2 E\right)}-
\sum_{p\mid N_E} \log p\sum_{e=1}^\infty \frac{\lambda_E(p^{2e})}{p^{e(1+2r)}}\nonumber\\
&+\sum_{p\nmid N_E}\log p\left[-\sum_{e=0}^\infty \frac{\lambda_E(p^{2e+2})-\lambda_E(p^{2e})}{p^{(e+1)(1+2r)}}+
\frac{1}{p
+1}\sum_{e=0}^\infty \frac{\lambda_E(p^{2e+2})-\lambda_E(p^{2e})}{p^{(e+1)(1+2r)}}\right]. 
\label{Aalpharr}
\end{align}

Recall that when $p\mid N_E$, we have $\alpha_E(p)=\lambda_E(p)$. Hence we find, using \eqref{zetaoverzeta}, $\eqref{lsymsderiv}$ 
and the identity $\lambda_E(p^{2e})-\lambda_E(p^{2e-2})=\alpha_E(p)^{2e}+\beta_E(p)^{2e}$ (for $p\nmid N_E$), that $\eqref{Aalpharr}$ becomes
\begin{equation*}
A_{\alpha,E}(r,r)=-\sum_{p\mid N_E}\sum_{\ell=1}^\infty 
\frac{\log p}{p^{\ell(1+2r)}}+\sum_{p\nmid N_E}\frac{\log p}{p+1}
\sum_{\ell=1}^\infty \frac{\alpha_E(p)^{2\ell}+\beta_E(p)^{2\ell}}{p^{\ell(1+2r)}}.
\end{equation*}
Combining this with the proof of Theorem \ref{oneleveldensityresult} and Remark \ref{RATIOSONELEV2}, we obtain 
\begin{align} \nonumber
\frac{1}{2\pi}&\int_{\mathbb R}
\phi\left(\frac{tL}{2\pi}\right)\bigg[\frac{\Gamma'(1+it) }{\Gamma(1+it) }
 +\frac{\Gamma'(1-it)}{\Gamma(1-it)} + 2 \bigg(-\frac{\zeta'(1+2it)}{\zeta(1+2it)} +\frac{L'\left(1+2it,{\rm Sym}^2E\right)}{L\left(1+2it,{\rm Sym}^2E\right)}\\&+A_{\alpha,E}(it,it)\bigg)- \frac{1}{it} \bigg]dt+  \frac{\phi(0)}{2}=-\frac{2} L\int_0^{\infty} \left( \frac{\widehat \phi(x/L) e^{-x} }{1-e^{-x}} - \widehat \phi(0) \frac{e^{-x}}x \right)dx\nonumber\\
&+\frac{1}{\pi i}\int_{(c')}\bigg[-\frac{\zeta'(1+2r)}{\zeta(1+2r)}+\frac{L'\left(1+2r,{\rm Sym}^2E\right)}{L\left(1+2r,{\rm Sym}^2E\right)}-\sum_{p\mid N_E}\sum_{\ell=1}^\infty 
\frac{\log p}{p^{\ell(1+2r)}}\nonumber\\
&+\sum_{p\nmid N_E}\frac{\log p}{p
+1}\sum_{\ell=1}^\infty \frac{\alpha_E(p)^{2\ell}+\beta_E(p)^{2\ell}}{p^{\ell(1+2r)}}\bigg]\phi\left(\frac{ iLr}{2 \pi}\right)dr.\label{ANOTHERONELEVELDENSITY}
\end{align}

Next we consider the term 
\begin{multline}
\frac{1}{\pi i}\int_{(c')}\bigg[-\frac{\zeta'(1+2r)}{\zeta(1+2r)}+\frac{L'\left(1+2r,{\rm Sym}^2E\right)}{L\left(1+2r,{\rm Sym}^2E\right)}-\sum_{p\mid N_E}\sum_{\ell=1}^\infty 
\frac{\log p}{p^{\ell(1+2r)}}\\
+\sum_{p\nmid N_E}\frac{\log p}{p
+1}\sum_{\ell=1}^\infty \frac{\alpha_E(p)^{2\ell}+\beta_E(p)^{2\ell}}{p^{\ell(1+2r)}}\bigg]\phi\left(\frac{ iLr}{2 \pi}\right)dr \label{theseven}
\end{multline}
appearing in \eqref{ANOTHERONELEVELDENSITY}. It follows from $\eqref{zetaoverzeta}$ and $\eqref{lsymsderiv}$ that we can rewrite $\eqref{theseven}$ as
\begin{equation*} 
\frac{1}{\pi i}\int_{(c')}
\bigg[-\sum_p \log p \sum_{\ell=1}^{\infty} \frac{\alpha_E(p)^{2\ell}+\beta_E(p)^{2\ell}}{p^{\ell(1+2r)}}+\sum_{p\nmid N_E}\frac{\log p}{p
+1}\sum_{\ell=1}^\infty \frac{\alpha_E(p)^{2\ell}+\beta_E(p)^{2\ell}}{p^{\ell(1+2r)}}\bigg]\phi\left(\frac{ iLr}{2 \pi}\right) dr.
\end{equation*}
Furthermore, making the substitution $u=-\frac{iLr}{2\pi}$, we obtain
\begin{multline}
\frac{2}{L}\int_{\mathcal C'}
\bigg[-\sum_p \log p \sum_{\ell=1}^{\infty} \frac{\alpha_E(p)^{2\ell}+\beta_E(p)^{2\ell}}{p^{\ell}}e^{-(\log p)\left(\frac{4\pi i u\ell}{L}\right)}\\
+\sum_{p\nmid N_E}\frac{\log p}{p
+1}\sum_{\ell=1}^\infty \frac{\alpha_E(p)^{2\ell}+\beta_E(p)^{2\ell}}{p^{\ell}}e^{-(\log p)\left(\frac{4\pi i u\ell}{L}\right)}\bigg]\phi\left(u\right) du,\label{theseventwo}
\end{multline}
where $\mathcal C'$ denotes the horizontal line $\Im(u)=-\frac{Lc'}{2\pi}$. On $\mathcal C'$ the summations inside the integral in $\eqref{theseventwo}$ converge absolutely and uniformly on compact subsets. Hence we can interchange the order of integration and summation and we have that $\eqref{theseventwo}$ becomes
\begin{align*}
&-\frac{2}{L}\sum_p \log p \sum_{\ell=1}^{\infty} \frac{\alpha_E(p)^{2\ell}+\beta_E(p)^{2\ell}}{p^{\ell}}\int_{\mathcal C'}\phi\left(u\right)e^{-2\pi i u\left(\frac{2\ell\log p}{L}\right)}du\\
&+\frac{2}{L}\sum_{p\nmid N_E}\frac{\log p}{p
+1}\sum_{\ell=1}^\infty \frac{\alpha_E(p)^{2\ell}+\beta_E(p)^{2\ell}}{p^{\ell}}\int_{\mathcal C'}\phi\left(u\right)e^{-2\pi i u\left(\frac{2\ell\log p}{L}\right)}du.
\end{align*}
Finally, we change the contour of integration from $\mathcal C'$ to the line $\Im(u) =0$. This is possible since we are assuming that $\widehat \phi$ has compact support  on $\R$ and since the entire function $\phi(z):=\int_{\R}\widehat\phi(x)e^{2\pi ixz}\,dx$ satisfies the estimate
\begin{align*}
\left|\phi(T+it)\right|\leq\frac1{2\pi|T|}\int_{\R}\big|\widehat\phi'(x)\big|\max(1,e^{xLc'})\,dx,
\end{align*}
uniformly for $-\frac{Lc'}{2\pi}\leq t\leq0$, as $T\to\pm\infty$. We conclude that \eqref{theseven} equals
\begin{align}
&-\frac{2}{L}\sum_p \log p \sum_{\ell=1}^{\infty} \frac{\alpha_E(p)^{2\ell}+\beta_E(p)^{2\ell}}{p^{\ell}}\widehat \phi\left(\frac{2\ell\log p}{L} 
\right)\nonumber\\
&+\frac{2}{L}\sum_{p\nmid N_E}\frac{\log p}{p
+1}\sum_{\ell=1}^\infty \frac{\alpha_E(p)^{2\ell}+\beta_E(p)^{2\ell}}{p^{\ell}}\widehat \phi\left( \frac{2\ell\log p}{L} 
\right). \label{lookslikesevenone}
\end{align}

\begin{lemma} \label{sumofpdivd}
Let $\varepsilon>0$ and let $p$ be a fixed prime. Then we have that 
\begin{equation}
 \sumtp =  \begin{cases} \frac{W(X)}{p+1} +O_{\varepsilon}\left((p^2XN_E)^{\varepsilon}\right) & \:{\rm if}\: p\nmid N_E, \\ 0 &\:{\rm otherwise}.\end{cases} \label{eqsumofpdivd}
\end{equation}
\end{lemma}

\begin{proof}
 It follows from Lemma \ref{lemmaweighted} (with $n=p^2$) and Remark \ref{whatiswx} that if $p\nmid N_E$, then
\begin{align*}
\sumtp=&\sumt - \underset{\substack{(d,N_E)=1\\ p\nmid d}}{\sum \nolimits^{*}}\widetilde w\left( \frac dX \right)
=\frac{W(X)}{p+1}+O_\varepsilon\left((p^2XN_E)^\varepsilon\right).
\end{align*}
\end{proof}

Combining $\eqref{lookslikesevenone}$ and Lemma \ref{sumofpdivd}, we have that $\eqref{theseven}$ becomes
\begin{align}
&-\frac{2}{L}\sum_p \log p \sum_{\ell=1}^{\infty} \frac{\alpha_E(p)^{2\ell}+\beta_E(p)^{2\ell}}{p^{\ell}}\widehat \phi\left( \frac{2\ell\log p}{L}\right) \nonumber\\
&+\frac {2}{LW(X)}\sum_{\ell=1}^\infty\sumt  \sum_{p\mid d} \frac{\left(\alpha_E(p)^{2\ell}+\beta_E(p)^{2\ell}\right) \log p}{p^{\ell}} \widehat \phi\left( \frac{2\ell\log p}{L} \right)\nonumber\\
&+O_{\varepsilon}\left(\frac{(XN_E)^\varepsilon}{LW(X)}\sum_{p\nmid N_E} \sum_{\ell=1}^\infty\frac{\left(\alpha_E(p)^{2\ell}+\beta_E(p)^{2\ell}\right) \log p}{p^{\ell-2\varepsilon}} \widehat \phi\left( \frac{2\ell\log p}{L} \right)\right).\label{theseventhree}
\end{align}
From the bounds $|\alpha_E(p)|, |\beta_E(p)| \leq 1$ and $W(X) \gg X$ (together with the assumption that $\widehat \phi$ has compact support), we have that the error term in $\eqref{theseventhree}$ is at most $O_{\varepsilon}(X^{-1+\varepsilon})$. Finally, noting that $$\left( \frac d{p^{2\ell}}\right)
=\begin{cases} 1 &\:{\rm if}\: p \nmid d, \\0&\:{\rm if}\: p \mid d,\end{cases}$$
we find that $\eqref{theseventhree}$ equals
\begin{equation*}
-\frac {2}{LW(X)}\sum_{\ell=1}^\infty \sum_{p} \frac{\left(\alpha_E(p)^{2\ell}+\beta_E(p)^{2\ell}\right) \log p}{p^{\ell}} \widehat \phi\left( \frac{2\ell\log p}{L} \right) \sumt \left( \frac d {p^{2\ell}}\right)
+O_{\varepsilon}(X^{-1+\varepsilon}),
\end{equation*}
which together with \eqref{ANOTHERONELEVELDENSITY} and Lemma \ref{lemmaweighted} and Remark \ref{whatiswx} (as in the proof of Theorem \ref{main theorem}) concludes the proof of Theorem \ref{extendedoneleveldensityresult}.

\end{document}